\documentclass[12pt, reqno]{amsart}
\usepackage{amssymb, amsthm, amsmath, amsfonts}
\usepackage{array, epsfig}
\usepackage{bbm, enumerate}
\usepackage{color}
\usepackage{hyperref,verbatim}

  \evensidemargin
\oddsidemargin
\begin{document}

\newcommand{\be}{\begin{equation}}
\newcommand{\ee}{\end{equation}}
\newcommand{\bea}{\begin{eqnarray}}
\newcommand{\eea}{\end{eqnarray}}
\newcommand{\beaa}{\begin{eqnarray*}}
\newcommand{\eeaa}{\end{eqnarray*}}
\newcommand{\Var}{\mathop{\mathrm{Var}}\nolimits}

\renewcommand{\proofname}{\bf Proof}
\newtheorem{conj}{Conjecture}
\newtheorem*{cor*}{Corollary}
\newtheorem{cor}{Corollary}
\newtheorem{proposition}{Proposition}
\newtheorem{lemma}{Lemma}

\newtheorem{taggedlemmax}{Lemma}
\newenvironment{taggedlemma}[1]
 {\renewcommand\thetaggedlemmax{#1}\taggedlemmax}
 {\endtaggedlemmax}
 
\newtheorem{theorem}{Theorem}

\theoremstyle{remark}
\newtheorem{remark}{Remark}
\newtheorem{example}{Example}

\newfont{\zapf}{pzcmi}

\def\PP{\mathrm{P}}
\def\EE{\mathrm{E}}

\def\R{\mathbb{R}}
\def\Q{\mathbb{Q}}
\def\Z{\mathbb{Z}}
\def\N{\mathbb{N}}
\def\E{\mathbb{E}}
\def\P{\mathbb{P}}
\def\V{\mathbb{D}}
\def\ZZ{\mathcal{Z}}
\def\XX{\mathcal{X}}
\def\BB{\mathcal{B}}
\def\I{\mathbbm{1}}
\newcommand{\D}{\hbox{\zapf D}}
\newcommand{\eps}{\varepsilon}
\newcommand{\sgn}{\mathop{\mathrm{sign}}\nolimits}
\newcommand{\eqdistr}{\stackrel{d}{=}}
\newcommand{\supp}{\mathop{\mathrm{supp}}\nolimits}
\newcommand{\Int}{\mathop{\mathrm{Int}}\nolimits}
\newcommand{\Inv}{\mathop{\mathrm{Inv}}\nolimits}
\newcommand{\Sym}{\text{\rm Sym}}
\newcommand{\dist}{\operatorname{dist}}
\newcommand{\tc}{\textcolor{red}}
\renewcommand{\Re}{\operatorname{Re}}
\renewcommand{\mod}{\operatorname{mod}}
\newcommand{\Mod}[1]{\ \mathrm{mod}\ #1}

\newcommand{\overbar}[1]{\mkern 1.5mu\overline{\mkern-3mu#1\mkern-3mu}\mkern 1.5mu}
\newcommand{\todistr}{\overset{d}{\underset{n\to\infty}\longrightarrow}}
\newcommand{\toprobab}{\overset{\P}{\underset{n\to\infty}\longrightarrow}}

\title{Stability of overshoots of zero mean random walks}
\author{Aleksandar Mijatovi\'c} 
\address{Aleksandar Mijatovi\'c, Department of Statistcis, University of Warwick \& The Alan Turing Institute}
\email{a.mijatovic@warwick.ac.uk}

\author{Vladislav Vysotsky}
\address{Vladislav Vysotsky, Department of Mathematics, University of Sussex
and St.\ Petersburg Department of Steklov Mathematical Institute}
\email{v.vysotskiy@sussex.ac.uk}

\begin{abstract}
We prove that for a random  walk on the real line whose increments
have zero mean and are either integer-valued or spread out (i.e.\ the distributions of the steps of the walk are eventually non-singular), 
the Markov chain of overshoots above a fixed level converges 
in total variation to its stationary distribution. 
We find the explicit form of this distribution heuristically and then prove its invariance using a time-reversal argument.
If, in addition, the increments of the walk are in the domain of attraction of a non-one-sided $\alpha$-stable law with 
index $\alpha\in(1,2)$ (resp.\ have finite variance), we establish geometric (resp.\ uniform) ergodicity for the Markov chain of overshoots. All the convergence results above are also valid for the Markov chain obtained by sampling the walk at the entrance times into an interval.
\end{abstract}

\maketitle

\section{Introduction}
\label{sec:Intro}

Let $S= (S_n)_{n \ge 0}$ with $S_n=
S_0 + X_1+ \ldots + X_n$ be a one-dimensional random walk with independent
identically distributed (i.i.d.) increments $X_1, X_2, \ldots$ and the starting
point $S_0$ that is a random variable independent with the increments.  
Assume that 
\begin{equation}
\label{eq:Assumption}
\E|X_1|\in(0,\infty)\qquad\text{ and }\qquad \E X_1=0,
\end{equation}
which implies that $\limsup_{n\to\infty} S_n = -\liminf_{n\to\infty} S_n = \infty$ a.s. 
Define the {\it up-crossings times} of zero 
\begin{equation}
\label{eq:crossing_time_def}
T_0:=0, \quad T_{n}:= \inf\{k>T_{n-1}: S_{k-1} < 0, S_k \ge 0 \}, \qquad n \in\N,
\end{equation}
and let
\begin{equation}
\label{eq:Chain_definitions}
O_n:=S_{T_n}, \qquad U_n:= S_{T_n - 1}, \qquad n \in \N; 
\end{equation}
be the corresponding {\it overshoots} and {\it undershoots}; put
$O_0=U_0:=S_0$. The choice of zero is arbitrary and can be
replaced by any fixed level.
The sequence of overshoots $O=(O_n)_{n \ge 0}$  is a Markov
chain. The sequence of undershoots $U=(U_n)_{n \ge 0}$
also forms a Markov chain. Both statements can be checked easily, although the latter one is less intuitive. We are mostly interested 
in the chain of overshoots,  but our techniques also yield results for the chain of undershoots. 

Under assumption~\eqref{eq:Assumption}, consider the law
$$
\pi_+(dx):= \frac{2}{\E|X_1|} \I_{[0, \infty)}(x) \P(X_1 > x) \lambda(dx), \qquad x \in \ZZ,
$$
where $\ZZ$ is the {\it state space} of the walk $S$, defined as the minimal closed (in the topological sense)
subgroup of $(\R, +)$ containing the topological support of the distribution of
$X_1$, and $\lambda$ is the Haar measure on $(\ZZ, +)$ normalized such that $\lambda([0,x) \cap \ZZ)=x$ for positive $x \in \ZZ$. 

We will prove that the distribution $\pi_+$ is invariant for the Markov chain of overshoots $O$ (Theorem~\ref{thm: stationary distr}). Our proof is 
based on a time reversal of the path of  $S$ between the up-crossings of the
level zero.  Since this proof gives no  insight into the form of $\pi_+$, in
Section~\ref{Sec: Derivation} we  present a heuristic argument which we used to
find this invariant distribution. The invariance of $\pi_+$ is also established
in our companion paper~\cite[Corollary to Theorem~3]{MijatovicVysotskyMC} in a much more general
setting using entirely different methods based on infinite ergodic theory; the
proof presented here precedes the one in~\cite{MijatovicVysotskyMC}.
By~\cite[Corollary to Theorem~4]{MijatovicVysotskyMC}, the assumption in~\eqref{eq:Assumption}
implies that 
the law $\pi_+$ is a unique (up to multiplicative constant) locally finite Borel invariant measure of the chain of overshoots $O$ on $\ZZ$. Moreover, we will see in  Section~\ref{sec: Notation} that assumption~\eqref{eq:Assumption} is the weakest possible ensuring that $O$ has an invariant distribution (i.e.\ probability measure).

The main goal of this paper is to study convergence of the Markov chain of overshoots $O$ to its unique invariant law $\pi_+$. Our  aim is to identify the conditions on the law of the increments of $S$
under which the total variation distance between the law of $O_n$ and $\pi_+$ converges to zero as $n \to \infty$
(Theorem~\ref{thm: convergence general}) and study its rate of decay (Theorem~\ref{thm: geometric ergodicity}).
Since  the chain $O$
is in general neither
weak Feller (\cite[Remark~5]{MijatovicVysotskyMC}) nor $\psi$-irreducible (see 
Section~\ref{sec:Concluding_remarks}), 
the total variation convergence requires additional smoothness assumptions on the distribution of increments of $S$.
In particular, Theorem~\ref{thm: convergence general} holds if the distribution of $X_1$ is either arithmetic or spread out, which means respectively that either $X_1$ is supported on $d\Z$ for some $d>0$ or  the distribution of $S_k$  is non-singular for some $k \ge 1$.  The geometric rate of convergence in 
Theorem~\ref{thm: geometric ergodicity}
is established under a further assumption that the  law of $X_1$ is in the domain of attraction of a non-one-sided $\alpha$-stable law with index $\alpha \in (1,2)$. For increments with finite variance we get a stronger version with the geometric rate of convergence uniformly in the starting point of $O$. Section~\ref{sec:Concluding_remarks} concludes the paper by offering a conjecture about the weak convergence of the Markov chain $O$ to $\pi_+$ without additional assumptions on the law of $X_1$ other than~\eqref{eq:Assumption}.

Our interest in the Markov chains of overshoots of random walks stems from
their close connection to the local time of the random walk at level zero (see
Perkins~\cite{Perkins}) and the fact that they appear in the study of  the
asymptotics of the probability that the integrated random walk $(S_1 + \ldots
+S_k)_{1 \le k \le n}$ stays positive (see Vysotsky~\cite{Vysotsky2010,
Vysotsky2014}). A detailed discussion with applications and further connections
to a special class of Markov chains called {\it random walks with switch} at
zero, is available in~\cite[Section~1.2]{MijatovicVysotskyMC}. Let us mention
that distributions of the same form as $\pi_+$ appear on many occasions -- this
is discussed in~Sections~\ref{sec: Notation} and~\ref{sec: second}.

Finally, we note that our methods developed for establishing convergence of the
chain $O$ of overshoots above zero work without any changes for the Markov {\it
chain of entrances} into the interval $[0, h]$ with any $h>0$. In
Section~\ref{sec: entrance interval} we show that all the results for $O$
remain valid for this new chain, whose stationary distribution, given in~\eqref{eq:pi_h_explicit} below,
is unique and explicit; see also~\cite[Section~4]{MijatovicVysotskyMC}.

\section{Stationary distributions of overshoots} 
\label{sec:Stationary_Distrib}
\subsection{Setting and results} \label{sec: Notation}
Consider the random walk 
$S=(S_n)_{n \ge 0}$
from Section~\ref{sec:Intro}, and define its version $S'=(S_n')_{n \ge 0}$ with $S_n ' := S_n - S_0$, which always starts at zero.
We assume that $S$, as well as all the other random elements considered in this paper, are defined on a generic measurable space equipped with a variety of measures: a probability measure $\P$; the family of probability measures $\{\P_x\}_{x \in \R}$ given by  $\P_x(S \in \cdot) = \P(x+S' \in \cdot)$ (satisfying $\P_x(S_0=x)=1$); and the measures of the form $\P_\mu(\cdot)  := \int_\R \P_x(\cdot) \mu (dx)$, where $\mu$ is a Borel measure $\mu$ on $\R$. We do not necessarily assume that $\mu$ is a probability but we prefer to (ab)use the probabilistic notation $\P_\mu$ and the terms ``law'', ``expectation'', ``random variable'', etc., by which we actually mean the corresponding notions of general measure theory. Under the measure $\P_\mu$, the starting point $S_0$ of the random walk $S$ follows the ``law'' $\mu$. Denote by $\E$ and $\E_x$ the respective expectations under $\P$ and $\P_x$. 
All the measures on topological spaces considered in the paper are Borel, that is defined on the corresponding Borel $\sigma$-algebras.

Recall that the state space $\ZZ$ of the random walk $S$ was defined as the
minimal closed subgroup of $(\R, +)$ containing the support of the distribution
of $X_1$. Let us give a different representation for $\ZZ$ assuming throughout that $X_1$ is not degenerate. For any
$h\in[0,\infty)$, let $\ZZ_h$ be the real line $\R$ if $h=0$ and the integer
lattice $\Z$ multiplied by $h$ if $h>0$: $$\ZZ_h:=
\begin{cases}
\R, &  \text{if } h=0,\\
h\Z, & \text{if } h>0.\\
\end{cases}
$$
We equip $\ZZ_h$ with the discrete (resp.\ Euclidean) topology if $h>0$ (resp.\ $h=0$). 
Note that any closed (in the topological sense) subgroups of $(\R, +)$ is of the form
$(\ZZ_h,+)$ for some $h\geq0$. 
Finally, denote 
$\ZZ_h^+:=\ZZ_h \cap [0, \infty)$ and $\ZZ_h^-:=\ZZ_h \cap (-\infty,0)$.

Define the \textit{span} 
of the distribution of increments of $S$
by
\begin{equation}
\label{eq:span}
d:=\sup\{h\in[0,\infty): \P(X_1\in \ZZ_h)=1 \},
\end{equation}
and note that $d\in[0,\infty)$ and $\ZZ = \ZZ_d$.  
We {\it always} assume that the random walk starts in $\ZZ_d$, hence
$\P(S_0, S_1, \ldots \in\ZZ_d)=1.$ 
The distribution of increments of $S$ is called {\it arithmetic} (with span $d$) if $d>0$ and is called {\it non-arithmetic} if $d=0$. 
We shall often use $d>0$ and $d=0$ as synonyms for arithmetic and non-arithmetic, respectively. 
Define the
measure $\lambda_d$ on $\ZZ_d$ as follows: for any $B \in \mathcal{B}(\ZZ_d)$,
put $$
\lambda_d(B):=
\begin{cases}
\lambda_0(B), &  \text{if } d=0,\\
d \cdot \# B, & \text{if } d>0,\
\end{cases}
$$
where $\lambda_0$  denotes the Lebesgue measure on $\R$ and $\#$ denotes the number of elements in a set. 
Then $\lambda_d$ is the normalized Haar measure on the additive group $\ZZ_d=\ZZ$, as defined in the Introduction.
Define the measures $\lambda_d^+(dy):= \I_{\ZZ_d^+}(y) \lambda_d(dy)$ and $\lambda_d^-(dy):= \I_{\ZZ_d^-}(y) \lambda_d(dy)$ on $\ZZ_d$. Put
\begin{equation}
\label{eq:pi_pm_general}
\pi_+(dx) := c_1 \P(X_1 > x) \lambda_d^+(dx) \quad \text{and} \quad
\pi_-(dx) := c_1 \P(X_1 \le x) \lambda_d^-(dx), 
\qquad x \in \ZZ_d,
\end{equation}
where $c_1 :=1 $ if $\E |X_1| = \infty$ and $c_1 := 2/\E|X_1|$ if $\E |X_1| < \infty$. This extends the definition of $\pi_+$ given in the Introduction under assumption~\eqref{eq:Assumption}.

The classic trichotomy
states that the (non-degenerate) random walk $S$ either drifts to $+\infty$, drifts to $-\infty$, or {\it oscillates}; see Feller~\cite[Section~XII.2]{Feller}. By definition, the latter possibility means that
$\limsup_{n \to \infty} S_n=\infty$ a.s.\ and $\liminf_{n \to \infty } S_n =
-\infty$ a.s.  It is known that $S$ oscillates if and only if either $\E X_1=0$ and $\E |X_1| \in (0, \infty)$ or $\E X_1$ does not exist, i.e.\ $\E X_1^+=\E X_1^-=+\infty$, where $x^+:= \max\{x, 0\}$ and $x^-:=(-x)^+$ for a real $x$; cf.\ Feller~\cite[Theorems~XII.2.1]{Feller} and Kesten~\cite[Corollary~3]{Kesten}. 
In particular, oscillation holds when the random walk $S$ is {\it topologically
recurrent} on~$\ZZ_d$, which means that $\P_0(S_n \in G \text{ i.o.})=1$ for
every open neighbourhood $G \subset \ZZ_d$ of $0$. This is because such random walks
satisfy $\P_0(S_n \in G \text{ i.o.})=1$ for {\it every} non-empty open set $G
\subset \ZZ_d$; see Guivarc'h et al.~\cite[Theorem~24]{French}.

Clearly, oscillation is necessary and sufficient  for $S$ to cross a level infinitely often a.s., in which case the Markov chains of overshoots and undershoots of the zero level introduced \eqref{eq:crossing_time_def} and \eqref{eq:Chain_definitions} are well-defined.
Similarly, define the {\it down-crossings times} of the level zero 
$$T_0^\downarrow:=0, \quad T_n^\downarrow:=
\inf\{k>T_{n-1}^\downarrow: S_{k-1} \ge 0, S_k < 0 \}, \qquad n \in \N,$$ 
and the corresponding overshoots and undershoots at the down-crossings 
\begin{equation} \label{eq: chains downcrossing} 
O_n^\downarrow=S_{T_n^\downarrow}, \qquad U_n^\downarrow:= S_{T_n^\downarrow - 1}, \qquad n \in \N
\end{equation}
with $O_0^\downarrow=U_0^\downarrow:=S_0$. The random sequences in~\eqref{eq:Chain_definitions} and \eqref{eq: chains downcrossing} are
defined on the event that all crossing times $T_n$ are finite. 
Since $S$ oscillates,
this event occurs almost surely under $\P$ and under $\P_\mu$ with arbitrary measure $\mu$
on $\ZZ_d$.  

The Markov chains of overshoots at up-crossings $O=(O_n)_{n \ge 0}$ and at
down-crossings $O^\downarrow=(O_n^\downarrow)_{n \ge 0}$ take values in
$\ZZ_d^+$ and $\ZZ_d^-$, respectively. Both chains are
started at  $O_0=O_0^\downarrow=S_0 \in \ZZ_d$. Note that there is asymmetry at zero. Namely, since $-\ZZ_d^+ \neq \ZZ_d^-$,
the down-crossing times $T_n^\downarrow$ (resp.\ positions $O_n^\downarrow$ and
$U_n^\downarrow$) need \textit{not} be equal to the up-crossing times $T_n$
(resp.\ positions $-O_n$ and $-U_n$) for the dual random walk $(-S_n)_{n \ge 0}$. Our consideration mostly
concerns $O$, which for brevity will be called the chain of overshoots if
there is no risk of confusion with $O^\downarrow$.

\begin{theorem}
\label{thm: stationary distr}
Let $S$ be any random walk that oscillates. Then the measure $\pi_+$ is
invariant for the Markov chains $O$ and $(-U_n - d)_{n \ge 0}$ of overshoots and  shifted sign-changed 
undershoots at up-crossings of the zero level, i.e.\ $\P_{\pi_+} (O_n \in \cdot) = \pi_+$ and $\P_{\pi_+} (-U_n - d \in \cdot) = \pi_+$ for all $n\in\N$. 
Similarly, $\pi_-$ is an invariant measure for the chains $O^\downarrow$ and $(-U_n^\downarrow-d)_{n \ge 0}$.
\end{theorem} 

\begin{remark} \label{rem: Plus_Minus}
We will show 
in Section~\ref{Sec: time-reversal} below
that the laws of overshoots and undershoots of the zero level at consecutive down- and up-crossings are related as follows:
$$
\P_{\pi_+} (O_1^\downarrow \in \cdot) = \pi_-, \quad \P_{\pi_-} (O_1 \in \cdot) = \pi_+,\quad \P_{\pi_+} (-U_1^\downarrow - d \in \cdot)=\pi_-,
\quad \P_{\pi_-} (-U_1- d \in \cdot)=\pi_+.
$$
\end{remark}

We will prove these results using an argument based on a time reversal of the path of $S$  between the up-crossings of the level zero. Since this proof gives no  insight about the form of $\pi_+$,  we will also present a heuristic argument which we used to find this invariant distribution. After these results were obtained, we found an entirely different proof of Theorem~\ref{thm: stationary distr}, which is based on the methods of infinite ergodic theory and applies in a much more general setting; see our companion paper~\cite{MijatovicVysotskyMC}. 

The assumption that the random walk $S$ oscillates is the weakest possible to consider the Markov chains of overshoots and undershoots. By~\cite[Corollary to Theorem~4]{MijatovicVysotskyMC}, the chains of overshoots and undershoots of such random walks possess no other locally finite invariant Borel measures (up to a multiplicative constant), including the ones singular with respect to $\pi_+$ and $\pi_-$. Therefore, the probabilistic question of convergence of these chains to stationarity can be posed only if 
the measures $\pi_+$ and $\pi_-$ in  Theorem~\ref{thm: stationary distr}
have total mass one. This 
need not be the case in general since every non-degenerate symmetric random walk oscillates.  
However,
by~\eqref{eq:pi_pm_general}, 
both measures
$\pi_+$ and $\pi_-$ have finite mass
if and only if $\E |X_1|\in(0,\infty)$, in which case the oscillation assumption forces $\E X_1 =0$ 
and the equalities 
$\pi_+(\ZZ_d)=\pi_-(\ZZ_d)=1$ follow. 
Thus, condition~\eqref{eq:Assumption} is the weakest assumption under which  convergence to stationarity
of the chains of overshoots and undershoots can be stated. 

Probability measures of the same form as $\pi_+$ and $\pi_-$ appear as limit distributions for the following stochastic processes closely related to random walks. Assume that \eqref{eq:Assumption} holds. First, $\pi_+$ is the unique stationary distribution of the {\it reflected random walk} driven by an i.i.d.\ sequence with the common non-arithmetic distribution $\P(X_1 \in \cdot | X_1 >0)$; see Feller~\cite[Section VI.11]{Feller} and Knight~\cite{Knight}. Second, $\frac12 \pi_+ + \frac12 \pi_-$  is an invariant distribution of a Markov chain whose increments are distributed as $\P(X_1 \in \cdot | X_1 <0)$ for all starting points in $\ZZ_d^+$ and as $\P(X_1 \in \cdot | X_1 >0)$ for all starting points in $\ZZ_d^-$. This chain belongs to a special type of Markov chains which we call {\it random walks with switch} at zero; see Borovkov~\cite{Borovkov} and cf.\ Vysotsky~\cite{Vysotsky2018}. Third, $\pi_+$ is known  as the stationary distribution, as well as the limit distribution, for the non-negative {\it residual lifetime} in a {\it renewal process} with inter-arrival times distributed according to $\P(X_1 \in \cdot | X_1 >0)$; see Asmussen~\cite[Section~V.3.3]{Asmussen} or Gut~\cite[Theorem~2.6.2]{Gut}. For random walks this limit property can be interpreted as follows.

Denote by $H_1^-$ the first strict descending ladder height of the random walk $S'$, i.e.\ the first strictly negative value of $S'$. Similarly, denote by $H_1^+$ the first strict increasing ladder height of $S'$. It is known that random variables $H_1^+$ and $H_1^-$ are integrable if $\E X_1 =0$ and $\E X_1^2 <\infty$; see Feller~\cite[Sections XVIII.4 and 5]{Feller}. When this is the case, by the results of renewal theory (e.g.\ by~\cite[Theorem~2.6.2]{Gut} and \eqref{eq:residual time form} below), we have
\begin{equation} 
\label{eq: inf down}
\P_x(O_1^\downarrow \in dy) \overset{d}{\underset{x\to\infty, \, x \in \ZZ_d}\longrightarrow} \frac{1}{ -\E H_1^-} \P(H_1^- \le y) \lambda_d(dy), \quad y \in \ZZ_d^-.
\end{equation}
The r.h.s.'s of \eqref{eq: inf down} is referred to as the distribution of the overshoot of the walk $S$ above an ``infinitely remote'' level at $-\infty$. This distribution equals $\pi_-$ defined for $H_1^-$ instead of $X_1$. Similarly, $\pi_+$ corresponds to the non-strict overshoot of $S$ above an ``infinitely remote'' level at $+\infty$, which is distributed as the strict overshoot above this level decreased by $d$:
\begin{equation} 
\label{eq: inf up}
\P_x(O_1 \in dy) \overset{d}{\underset{x\to -\infty, \, x \in \ZZ_d}\longrightarrow} \frac{1}{\E H_1^+} \P(H_1^+ > y) \lambda_d(dy) , \quad y \in \ZZ_d^+.
\end{equation}

\subsection{An alternative representation for $\pi_+$ and $\pi_-$} \label{sec: second}
Notice that the invariant measures $\pi_+$ and $\pi_-$ in Theorem~\ref{thm: stationary distr}, which are probabilities if and only if $\E |X_1| \in (0, \infty)$, are defined only in terms of the tails of the distribution of increments of the random walk $S$. On the other hand, it is natural to expect that $\pi_+$ and $\pi_-$ are closely related to the distributions of the overshoots above infinitely remote levels at $\pm \infty$ given by the r.h.s.'s of \eqref{eq: inf down} and \eqref{eq: inf up}. In this section we give such a representation. Note that since the limit distributions of overshoots above infinite levels exist only for zero-mean random walks with finite variance, before we proved Theorem~\ref{thm: stationary distr} it was not clear at all to us why the chain of overshoots should have a stationary distribution for walks with infinite variance.

Denote by $\tilde{H}_1^-$ the first non-strict (weak) descending ladder height of $S'$, i.e.\ the first non-positive value of $(S'_n)_{n \ge 1}$. 

\begin{lemma}
\label{lemma: 2nd formula}
For any random walk $S$ that oscillates, we have
$$
\pi_+ = c_1 \P(\tilde{H}_1^- \neq 0) \bigl[ \P(H_1^- \le x) \lambda_d^-(dx) \bigr ] * \P(H_1^+ \in \cdot) \quad \text{\emph{on} } \ZZ_d^+.
$$
Similarly,
$$
\pi_- = c_1 \P(\tilde{H}_1^- \neq 0) \bigl[ \P(H_1^+ > x) \lambda_d^+ (dx) \bigr ] * \P(H_1^- \in \cdot) \quad \text{\emph{on} } \ZZ_d^-.
$$
\end{lemma}
\begin{remark} \label{rem: interpretation}
If $\E X_1^2 < \infty$ (which implies \eqref{eq:Assumption}), the first identity can be interpreted as 
$$\pi_+ (dy) = \P( R^- + H_1^+ \in dy| R^- + H_1^+ \ge 0),$$ 
where $R^-$ is a random variable having the distribution of the overshoot of $S$ above an ``infinitely remote'' level at $-\infty$ (given by the r.h.s.\ of \eqref{eq: inf down}) and independent with $H_1^+$. Moreover,
$$
\P(R^- + H_1^+ \ge 0) = -\frac{1}{c_1 \P(\tilde{H}_1^- \neq 0) \E H_1^-} = -\frac{1}{c_1 \E \tilde{H}_1^-}.
$$
\end{remark}

Combining this with the analogous probabilistic interpretation of $\pi_-$ and in the case $d=0$ using that any distribution function is continuous a.e.\ with respect to the Lebesgue measure $\lambda_0$ allows us to rewrite the above representations directly in terms of the random walk as follows.

\begin{proposition} \label{prop: other pi}
For any random walk $S$ satisfying $\E X_1 =0$ and $0<\E X_1^2 < \infty$, we have 
$$
\P_x \Bigl(S_{T_1} \in \cdot \, \Bigl | \Bigr. S_{T_1^\downarrow} \ge S_{T_1^\downarrow +1}, \ldots, S_{T_1^\downarrow} \ge S_{T_1 -1} \Bigr) \overset{d}{\underset{x\to\infty, \, x \in \ZZ_d}\longrightarrow} \pi_+.
$$
and
$$
\P_x \Bigl(S_{T_1^\downarrow} \in \cdot \, \Bigl | \Bigr. S_{T_1} \le S_{T_1 +1}, \ldots, S_{T_1} \le S_{T_1^\downarrow -1}  \Bigr) \overset{d}{\underset{x\to-\infty, \, x \in \ZZ_d}\longrightarrow} \pi_-.
$$
\end{proposition}

The  representations for $\pi_+$ and $\pi_-$ in Lemma~\ref{lemma: 2nd formula} were found in~\cite{Vysotsky2018} by considering the overshoots above zero for the Markov chain of the so-called switching ladder heights, which is a particular example of random walks with switch at zero. Here we give a different independent proof.

\begin{proof}[{\bf Proof of Lemma~\ref{lemma: 2nd formula}.}]
From the Wiener--Hopf factorization 
\begin{equation} \label{Wiener-Hopf}
\P(X_1 \in \cdot) = \P(H_+ \in \cdot) + \P(\tilde H_1^- \in \cdot) - \P(H_1^+ \in \cdot) * \P(\tilde H_1^- \in \cdot)
\end{equation}
(Feller~\cite[Chapter XII.3]{Feller}) it follows that for any $y \in \ZZ_d^+$,
$$
\P(X_1 > y) = \P(H_1^+ > y) - \int_{(y, \infty)} \P(\tilde H_1^- > y - z) \P(H_1^+ \in dz) 
= \int_{(y, \infty)} \P(\tilde H_1^- \le y - z) \P(H_1^+ \in dz). 
$$
Then from the identity $\P(\tilde{H}_1^- \le u) = \P(\tilde{H}_1^- \neq 0) \P(H_1^- \le u)$ for $u \in \ZZ_d^-$, we get
$$
c_1 \P(X_1 > y) = c_1 \P(\tilde{H}_1^- \neq 0) \int_{\ZZ_d} \P( H_1^- \le y - z) \I_{\ZZ_d^-}(y-z) \P(H_1^+ \in dz).
$$
The l.h.s.\ is the density of $\pi_+$ with respect to $\lambda_d$, and for the r.h.s.\ it remains to use the following formula \eqref{eq: convolution density} for the density of convolutions. For any random variable $H$ supported on $\ZZ_d$ and any measure $\mu$ on $\ZZ_d$ with a bounded density $g$ with respect to $\lambda_d$, we have
\begin{equation} \label{eq: convolution density}
(\mu * \P(H\in \cdot)) (dx) = [\E g(x - H)]\lambda_d(dx), \qquad x \in \ZZ_d.
\end{equation}
This is evident for $d>0$. For the absolutely continuous case $d=0$, see e.g.\ Cohn~\cite[Proposition~10.1.12]{Cohn}.
\end{proof}

\subsection{Derivation of $\pi_+$} \label{Sec: Derivation}
Let us present a simple probabilistic argument that we used to {\it guess} the form of~$\pi_+$. 
Assume that $\E X_1=0$, the variance of increments $\sigma^2= \E X_1^2$ is finite and positive, and the random walk $S$ is integer-valued and aperiodic, i.e.\ the distribution of $X_1 - a$ is arithmetic with span $1$ for every $a \in \Z$. In this case $\ZZ_d^+ = \N_0$, where $\N_0:= \N_0 \cup \{0\}$.

Consider the number of up-crossings of the zero level by time $n$:
$$
L_n^\uparrow:=\sum_{i=0}^{n-1} \I(S_i <0, S_{i+1} \ge 0) = \max \{k \ge 0: T_k \le n\}.
$$ 
{\it Assume} that the chain $O$ has an ergodic stationary distribution $\mu$. 
Then by the ergodic theorem, for any $x, y \in \{z \in \N_0: \P(X_1>z)>0\}$,
\begin{equation} \label{eq: ergodic derivation}
\lim_{n \to \infty} \frac{1}{n} \sum_{i=1}^n \I(O_i = y) = \lim_{n \to \infty} \frac{1}{L_n^\uparrow} \sum_{i=1}^{L_n^\uparrow} \I(O_i = y) = \mu(y), \quad \P_x \text{-a.s.}
\end{equation}

On the other hand,
\begin{align*}
\E_x  \biggl[\frac{L_n^\uparrow}{\sqrt{n}} \cdot \frac{1}{L_n^\uparrow} \sum_{i=1}^{L_n^\uparrow} \I(O_i = y) \biggr] &= \frac{1}{\sqrt{n}}\sum_{i=1}^{n-1} \P_x(S_i <0, S_{i+1} = y) \\
&= \frac{1}{\sqrt{n}}\sum_{i=1}^{n-1} \sum_{k=1}^\infty \P_x(S_i = - k) \P( X_1 = y+k) \\
&= \sum_{k=1}^\infty \P( X_1 = y+k)  \frac{1}{\sqrt{n}} \sum_{i=1}^{n-1}  \P_x(S_i = - k). 
\end{align*}
By the local central limit theorem, there exists a constant $c>0$ such that for every integer $i$ and $k \ge 1$ we have $\P_x(S_i = - k) \le c / \sqrt{n}$, and also $\P_x(S_i = - k) \sim \frac{1}{\sqrt{2 \pi i} \sigma}$ as $i \to \infty$. Hence from \eqref{eq: ergodic derivation} and the dominated convergence theorem, we obtain
$$
\mu(y) \lim_{n \to \infty} \E_x \left [ \frac{L_n^\uparrow}{\sqrt{n}} \right ] = \sum_{k=1}^\infty \P( X_1 = y+k) \Bigl( \lim_{n \to \infty} \frac{1}{\sqrt{n}} \sum_{i=1}^{n-1} \frac{1}{\sqrt{2 \pi i} \sigma} \Bigr)= \sqrt{\frac{2}{\pi \sigma^2}} \P( X_1 > y).
$$
Thus, $\mu = \pi_+$ in the special case considered above. 

Therefore it is feasible that the distribution $\pi_+$ is stationary for the chain of overshoots $O$ for general random walks but of course we need to prove this directly, and even for the case considered here.

\subsection{Proof of Theorem~\ref{thm: stationary distr}} \label{sec:Stationary_distribution_via_reversibility}
The main result of the section, Proposition~\ref{prop: time reversal} below, 
reveals a distributional symmetry hidden in the trajectory of
an arbitrary oscillating random walk, which is key for the proof of Theorem~\ref{thm: stationary distr}.

Define new Markov transition kernels $P$ and
$Q$ on $\ZZ_d$ as follows: 
\begin{equation}
\label{eq:P_and_Q_kernels}
P(x, dy) := \P_x(-U_1 - d\in dy ), \quad Q(x,dy):= \P(X_1-d \in dy + x | X_1-d \ge x), \quad x,y \in \ZZ_d,
\end{equation}
with the convention that $Q(x,dy) := \delta_0(dy)$ in the case when $\P( X_1-d
\ge x) = 0$; the choice of the delta measure is arbitrary and will not be
relevant for what follows. The kernel $P$ is defined in terms of
the sign-changed first undershoot $U_1$, given in~\eqref{eq:Chain_definitions} above, which is
shifted by $d$ to ensure that $-U_1 - d$ may take value zero in the
arithmetic case. The kernel $Q$ corresponds to up-crossings of the zero level 
by the walk $S$. Clearly, for every $x\in\ZZ_d$,
the transition probabilities $P(x, dy)$ and $Q(x, dy)$ 
are supported on $\ZZ_d^+$. 

The transition kernels of the Markov chains of overshoots $(O_n)_{n \ge 0}$ and shifted sign-changed undershoots  $(-U_n - d)_{n \ge 0}$
equal $PQ$ and $QP$, respectively. 
More precisely, for any probability measure $\mu$ on 
$\ZZ_d$ 
and any $n\in\N$, 
\begin{equation}
\label{eq:O_U_kernels}
\P_\mu(O_n \in dy ) = [\mu (P Q)^n](dy), \quad \P_\mu(-U_n -d \in dy) = [\mu P (QP)^{n-1}](dy), \quad y\in \ZZ_d.
\end{equation}
Here for any transition kernel $T$ on $\ZZ_d$, 
by $\mu T$ we denoted the measure on $\ZZ_d$ 
given by 
$\mu T(dy):=\int_{\ZZ_d}T(z,dy)\mu(dz)$,
and put
$T^0(x,dy) = \delta_x(dy)$.

In the arithmetic case, we clearly have the equality $\lambda_d(dx) \P(X_1-d\in dy+x) = \lambda_d(dy) \P(X_1-d\in dx+y)$
of measures on $\ZZ_d \times \ZZ_d$; we will also prove this identity for $d=0$. Combined with the equality of measures $\P(X_1 -d \ge z) \lambda_d^+(dz) = \pi_+(dz)$ on $\ZZ_d$, this implies that 
the transition kernel $Q$ is {\it reversible} with respect to $\pi_+$. 
Put differently,
the {\it detailed balance} condition
$$\pi_+(dx) Q(x, dy) = \pi_+(dy) Q(y, dx), \qquad x, y \in \ZZ_d$$  
holds true for the measures on 
$\ZZ_d \times \ZZ_d$ (which are supported on $\ZZ_d^+ \times \ZZ_d^+$). 
Surprisingly, the kernel $P$ shares the same property.
Put together, we have the following statement, which we will prove in full below  in Section~\ref{Sec: time-reversal}.

\begin{proposition}
\label{prop: reversibility}
For any random walk $S$ that oscillates, the kernels $P$ and $Q$ are reversible with respect to~$\pi_+$.
\end{proposition}

A direct corollary of this proposition is the invariance of the measure $\pi_+$
for the Markov chains $(O_n)_{n \ge 0}$ and $(-U_n - d)_{n \ge 0}$ asserted by
Theorem~\ref{thm: stationary distr}. A similar argument yields the invariance of
$\pi_-$ for the chains $(O_n^\downarrow)_{n \ge 0}$ and $(-U_n^\downarrow-d)_{n
\ge 0}$ (use~\eqref{eq:time_reversal_minus} from Section~\ref{Sec:
time-reversal} below and a kernel decomposition for these chains analogous
to~\eqref{eq:O_U_kernels}). Thus Theorem~\ref{thm: stationary distr}
follows from Proposition~\ref{prop: reversibility},
which in turn is a direct corollary of Proposition~\ref{prop: time reversal} (see Section~\ref{Sec: time-reversal}). 

\subsubsection{The time reversal argument} \label{Sec: time-reversal}
We now present a result concerning the entire trajectory of the random walk between up-crossings of the level zero. 
Our proof is based on a generalisation of the argument from 
Vysotsky~\cite[Lemma~1]{Vysotsky2014}. 
It may be regarded as an illustration of the conclusion of Remark~5 in~\cite[Section~5.2]{MijatovicVysotskyMC} on general state-space Markov chains.

\begin{proposition}
\label{prop: time reversal}
For any random walk $S$ that oscillates, for any $m \in\N$ we have 
\begin{multline} 
\label{eq: time reversal}
\P_{\pi_+} \bigl(( S_0, S_1, \ldots, S_{T_m-1}, 0, \ldots ) \in \cdot \bigr) \\
=\P_{\pi_+}\bigl( (-S_{T_m - 1} - d, -S_{T_m - 2} -d, \ldots, -S_0 - d, 0, \ldots)\in \cdot \bigr)
\end{multline} 
and
\begin{multline} 
\label{eq:reversal_plusminus}
\P_{\pi_+} \bigl(( S_0, S_1, \ldots, S_{T_m^\downarrow-1}, 0, \ldots ) \in \cdot \bigr) \\
=\P_{\pi_-}\bigl( (-S_{T_m - 1} - d, -S_{T_m - 2} -d, \ldots, -S_0 - d, 0 ,\ldots) \in \cdot \bigr). 
\end{multline} 
\end{proposition}

The choice of the value $0$ in the random sequences in~\eqref{eq: time reversal}
and~\eqref{eq:reversal_plusminus} is arbitrary and could be substituted by any constant.
However, we stress that the equalities in~\eqref{eq: time reversal} and~\eqref{eq:reversal_plusminus} 
cease to hold if this constant value is substituted by the remaining part of the path of $S$.
Note that~\eqref{eq: time reversal} can be stated more elegantly as
\begin{equation*} 
 \bigl( S_0, S_1, \ldots, S_{T_m-1}\bigr) \eqdistr \bigl( -S_{T_m - 1} -d, -S_{T_m - 2}-d, \ldots, -S_0 - d \bigr) \text{ under } \P_{\pi_+}.
\end{equation*}

\begin{remark} \label{rem: to time reversal}
Similarly, we have
\begin{multline} 
\label{eq:time_reversal_minus}
\P_{\pi_-} \bigl(( S_0, S_1, \ldots, S_{T_m^\downarrow-1}, 0, \ldots ) \in \cdot \bigr) \\
=\P_{\pi_-}\bigl( (-S_{T_m^\downarrow - 1} - d, -S_{T_m^\downarrow  - 2} -d, \ldots, -S_0 - d, 0, \ldots)\in \cdot \bigr)
\end{multline}  
and 
\begin{multline} 
\label{eq:reversal_minusplus}
\P_{\pi_-} \bigl(( S_0, S_1, \ldots, S_{T_m-1}, 0, \ldots ) \in \cdot \bigr) \\
=\P_{\pi_+}\bigl( (-S_{T_m^\downarrow  - 1} - d, -S_{T_m^\downarrow  - 2} -d, \ldots, -S_0 - d, 0 ,\ldots) \in \cdot \bigr). 
\end{multline} 
\end{remark}

We first prove two simple corollaries of Proposition~\ref{prop: time reversal}.

\begin{proof}[{\bf Proof of Propositions~\ref{prop: reversibility}}] 
Reversibility of the $P$-kernel follows immediately by~\eqref{eq: time reversal} with $m=1$ since $U_1 = S_{T_1 - 1}$.

As explained above, reversibility of the $Q$-kernel follows from
the equalities of measures
$$
\lambda_d(dx) \P(X_1-d\in dy+x) = \lambda_d(dy) \P(X_1-d\in dx+y)
$$
on $\ZZ_d \times \ZZ_d$ and  $\P(X_1 -d \ge z) \lambda_d^+(dz) = \pi_+(dz)$ on $\ZZ_d$. The latter equality is trivial. The former one is equivalent to
\begin{equation} \label{eq: reversibility plus}
\lambda_d(dx) \P(x + X_1 \in dy) = \lambda_d(dy) \P(y-X_1\in dx), \qquad x, y \in \ZZ_d,
\end{equation}
as follows from substituting $y$ by $y-d$ using the invariance of $\lambda_d$ under tshifts in $\ZZ_d$ and substituting $x$ by $-x$ using the central symmetry of $\lambda_d$; cf.~\eqref{eq: dual RW} below for the meaning of~\eqref{eq: reversibility plus}.  

It suffices to check the equality of measures \eqref{eq: reversibility plus} only for rectangular sets with Borel sides $A, B \subset \ZZ_d$. By Fubini's theorem and the mentioned shift invariance of $\lambda_d$,
\begin{align*}
\bigl[\lambda_d(dx) \P(x + X_1 \in dy)\bigr](A \times B) &= \int_{\ZZ_d^2} \I(x \in A, x + z \in B) \lambda_d(d x) \otimes \P(X_1 \in dz) \\
&=\int_{\ZZ_d} \lambda_d(A \cap (B - z)) \P(X_1 \in dz)\\
&=\int_{\ZZ_d} \lambda_d\bigl(B \cap (A- z)  \bigr) \P(-X_1 \in dz)\\
&=\bigl[\lambda_d(dx) \P(x - X_1 \in dy)\bigr](B \times A),
\end{align*}
where the last equality follows from the first two. This is exactly \eqref{eq: reversibility plus}. 
\end{proof}


Recall that Remark~\ref{rem: Plus_Minus} asserts that 
$$
\P_{\pi_+} (O_1^\downarrow \in \cdot) = \pi_-, \quad \P_{\pi_-} (O_1 \in \cdot) = \pi_+,\quad \P_{\pi_+} (-U_1^\downarrow - d \in \cdot)=\pi_-,
\quad \P_{\pi_-} (-U_1- d \in \cdot)=\pi_+.
$$

\begin{proof}[{\bf Proof of Remark~\ref{rem: Plus_Minus}}]
Fix $m=1$. By~\eqref{eq: time reversal}, 
the random variables 
$-O_1^\downarrow - d = - S_{T_1^\downarrow}-d$
and
$U_1^\downarrow  =  S_{T_1^\downarrow-1}$
have the same law under
$\P_{\pi_+}$,
hence
$\P_{\pi_+} (O_1^\downarrow \in \cdot) =
\P_{\pi_+} (-U_1^\downarrow - d \in \cdot)$. 
By~\eqref{eq:reversal_plusminus}, 
the law of 
$-U_1^\downarrow-d  = - S_{T_1^\downarrow-1}-d$
under 
$\P_{\pi_+}$
is the same as the law of $S_0$
under 
$\P_{\pi_-}$,
i.e.
$\P_{\pi_+} (-U_1^\downarrow - d \in \cdot)=\pi_-$,
and hence 
$\P_{\pi_+} (O_1^\downarrow \in \cdot) = \pi_-$. Similarly, by~\eqref{eq:time_reversal_minus}, we find
$\P_{\pi_-} (O_1\in \cdot) =
\P_{\pi_-} (-U_1- d \in \cdot)$.
Finally, by~\eqref{eq:reversal_minusplus}, we have
$\P_{\pi_-} (-U_1- d \in \cdot)=\pi_+$.
\end{proof}

We now prove the main statement of the section. 

\begin{proof}[{\bf Proof of Proposition~\ref{prop: time reversal}}]
Consider equality~\eqref{eq: time reversal} in the case $m=1$. Pick an arbitrary $k\in\N$ and define the time-reversal mapping $R_k:\R^{k+1}\to\R^{k+1}$ by 
$$R_k (x_0, \ldots, x_k):= (-x_k-d, \ldots, -x_0-d).$$ 
Introduce the random vector 
$K:=(S_0,\ldots,S_k)$ 
and note that~\eqref{eq: time reversal}
follows if 
we establish the equality of measures on $(\ZZ_d)^{k+1}$: 
\begin{equation} \label{eq: reversal k}
\P_{\pi_+}(K \in \cdot, T_1 = k+1 ) = \P_{\pi_+}(R_k(K) \in \cdot, T_1 = k+1).
\end{equation}

Put
$$\tilde{\ZZ}_d^+ := 
\begin{cases}
\ZZ_d^+ \setminus \{0\}, & \text{if } d=0, \\
\ZZ_d^+, & \text{if } d>0, \\
\end{cases}
$$
and denote $C_k:=\cup_{i=0}^{k-1} (\tilde{\ZZ}_d^+)^i \times (\ZZ_d^-)^{k-1-i}$. Then $C'_k:= \tilde{\ZZ}_d^+ \times C_k \times \ZZ_d^-$ is the set of sequences of length $k+1$ that start from $\tilde{\ZZ}_d^+$, down-cross the level zero exactly once, and in the non-arithmetic case have no zeroes. 

Note that $R_k$ is an invertible mapping on $\R^{k+1}$, and it is an involution. Further, $R_k(C_k)=C_k$ since $-\tilde{\ZZ}_d^+ - d = \ZZ_d^-$ in both cases $d=0$ and $d>0$. Similarly,  $R_k(C_k')=C_k'$, implying that $R_k(\R^{k+1} \setminus C_k') = \R^{k+1} \setminus C_k'$. This gives
\begin{equation} \label{eq: supported on C_k}
\P_{\pi_+} \bigl( R_k(K) \in \R^{k+1} \setminus C_k', T_1 = k+1 \bigr) = \P_{\pi_+} \bigl( K \in \R^{k+1} \setminus C_k', T_1 = k+1 \bigr)= 0.
\end{equation}
The second equality is trivial in the arithmetic case. In the non-arithmetic case, it is due to the fact that $K$ has density with respect to the Lebesgue measure on $\R^{k+1}$, which in turn holds true since in this case the measure $\pi_+$ has density with respect to the Lebesgue measure on $\R$.

By \eqref{eq: supported on C_k}, if suffices to check equality \eqref{eq: reversal k} on rectangles of the form $B_0 \times B \times B_k$ with Borel sides $B_0 \subset \tilde{\ZZ}_d^+, B_k \subset \ZZ_d^-$ and $B \subset C_k$. Using the definition of $\pi_+$ and the fact that $X_{k+1}$ is independent with $K$ under $\P_{x_0}$ for every $x_0 \in \ZZ_d$, we obtain
\begin{align}
\label{eq: conditioning on endpoints}
&\mathrel{\phantom{=}}  \P_{\pi_+} \bigl( K \in B_0 \times B \times B_k,  T_1 = k+1 \bigr) \notag \\
&= \int_{B_0} \P_{x_0} \bigl( (S_1, \ldots, S_k) \in B \times B_k,  T_1 = k+1 \bigr) \pi_+(d x_0) \notag\\
&= \int_{B_0} \int_{B_k} \Bigl [ \P_{x_0} \bigl( (S_1, 
\ldots, S_{k-1}) \in B, T_1 = k+1 \bigl| \bigr. S_k = x_k  \bigr) \P(X_1 > x_0 ) \Bigr ] \P_{x_0}(S_k \in dx_k) \lambda_d (d x_0) \notag \\
&= \int_{B_0 \times B_k} f_B(x_0, x_k) \P_{\lambda_d}((S_0, S_k) \in d x_0 \otimes d x_k),
\end{align}
where
$$
f_B(x_0, x_k):= \P_{x_0} \bigl( (S_1, \ldots, S_{k-1}) \in B \bigl| \bigr. S_k = x_k  \bigr) \P(X_1 > x_0 ) \P(X_1 \ge -x_k)
$$
for  $(x_0, x_k) \in \tilde \ZZ_d^+ \times \ZZ_d^-$. Then we use equality \eqref{eq: conditioning on endpoints} to get
\begin{align}
\label{eq: conditioning on endpoints R}
&\mathrel{\phantom{=}}\P_{\pi_+} \bigl( R_k(K) \in B_0 \times B \times B_k,  T_1 = k+1 \bigr) \notag \\
&= \P_{\pi_+} \bigl( K \in (-B_k-d) \times R_{k-2}(B) \times (-B_0-d), T_1 = k+1 \bigr) \notag \\
&= \int_{(-B_k-d) \times (-B_0-d)} f_{R_{k-2}(B)}(x_0, x_k) \P_{\lambda_d}((S_0, S_k) \in d x_0 \otimes d x_k) \notag \\
&= \int_{B_0 \times B_k} f_{R_{k-2}(B)}(R_1(x_0, x_k)) \P_{\lambda_d}\bigl(R_1(S_0, S_k) \in d x_0 \otimes d x_k \bigr), 
\end{align}
where in the last equality we used the change of variables formula, the fact that $R_k$ is an involution, and the equality $(-B_k-d) \times (-B_0-d) = R_1(B_0, B_k)$.

Let us simplify the integrand under the last integral in~\eqref{eq: conditioning on endpoints R}. We have
\begin{align*}
&\mathrel{\phantom{=}} \P_{-x_k-d} \bigl( (S_1, \ldots, S_{k-1}) \in R_{k-2}(B) \bigl| \bigr. S_k = -x_0-d  \bigr)\\
&= \P_0 \bigl( (S_1 - x_k -d, \ldots, S_{k-1} - x_k -d ) \in R_{k-2}(B) \bigl| \bigr. S_k = x_k-x_0 \bigr)  \\
&= \P_0 \bigl( R_{k-2}(S_1 - S_k -x_0 - d, \ldots, S_{k-1} - S_k -x_0-d ) \in B \bigl| \bigr. S_k = x_k-x_0 \bigr) \\
&= \P_0 \bigl( (S_k - S_{k-1} + x_0, \ldots, S_k - S_1 + x_0 ) \in B \bigl| \bigr. S_k + x_0 = x_k \bigr).
\end{align*}
The well-known duality principle for random walks states that the random vectors $(S_1,\ldots,S_k)$ and $(S_k-S_{k-1},\ldots,S_k-S_1,S_k)$ have the same law under $\P_0$. By a conditional version of this distributional identity, for every $x_0 \in \ZZ_d$ and $\P_{x_0} (S_k \in \cdot)$-a.e.\ $x_k \in \ZZ_d$,
\begin{equation} \label{eq: duality for RW bridge} \P_{-x_k-d} \bigl( (S_1, \ldots, S_{k-1}) \in R_{k-2}(B) \bigl| \bigr. S_k = -x_0-d  \bigr) = \P_{x_0} \bigl( (S_1, \ldots, S_{k-1}) \in B \bigl| \bigr. S_k = x_k \bigr).
\end{equation}
By the definition of $f_B$, this gives
\begin{align*}
f_{R_{k-2}(B)}(R_1(x_0, x_k))&= f_{R_{k-2}(B)}(-x_k-d, -x_0-d)) \\
&= \P_{x_0} \bigl( (S_1, \ldots, S_{k-1}) \in B \bigl| \bigr. S_k = x_k \bigr) \P(X_1 > -x_k -d ) \P(X_1 \ge -x_0 -d).
\end{align*}
Thus, using in the non-arithmetic case the fact that a distribution function can have at most countably many jumps, we get
\begin{equation} \label{eq: symmetic f}
f_B(x_0, x_k) = f_{R_{k-2}(B)}(R_1(x_0, x_k)), \qquad \P_{\lambda_d}((S_0, S_k) \in \cdot)\text{-a.e.\ } (x_0, x_k).
\end{equation}

Hence by \eqref{eq: conditioning on endpoints}, \eqref{eq: conditioning on endpoints R}, and \eqref{eq: symmetic f} combined with \eqref{eq: supported on C_k}, equality \eqref{eq: reversal k} will follow once we show the following equality of measures on $\ZZ_d^+ \times \ZZ_d^-$:
\begin{equation}
\label{eq: RW reversal}
\P_{\lambda_d}((S_0, S_k) \in \cdot) = \P_{\lambda_d}(R_1(S_0, S_k) \in \cdot).
\end{equation}
By translation invariance of $\lambda_d$ under shifts in $\ZZ_d$,
$$
\P_{\lambda_d}(R_1(S_0, S_k) \in \cdot) = \P_{\lambda_d}((-S_k-d, -S_0-d) \in \cdot)
=\P_{\lambda_d}((-S_k, -S_0) \in \cdot),
$$
and thus the claim \eqref{eq: RW reversal} reduces to
\begin{equation}
\label{eq: dual RW}
 \P_{\lambda_d}((S_0, S_k) \in \cdot) = \P_{\lambda_d}((-S_k, -S_0) \in \cdot),
\end{equation}
which means that the random walk $-S$ is dual to $S$ with respect to $\lambda_d$. To prove this property, note that by the shift invariance of $\lambda_d$ under  shifts in $\ZZ_d$, the equality \eqref{eq: dual RW} of measures on  $\ZZ_d^2=\ZZ_d \times \ZZ_d$ is equivalent to  $\P_{\lambda_d}((S_0, S_k) \in \cdot) = \P_{\lambda_d}((S_0-S_k', S_0) \in \cdot)$. This is exactly \eqref{eq: reversibility plus} with $X_1$ replaced by $S_k'$.


Thus,~\eqref{eq: time reversal} is proved for $m=1$. The general case $m\in\N$ follows analogously, with the only
difference that the set $C_k'$ shall account for $2m-1$ crossings of the level zero. 

Consider now~\eqref{eq:reversal_plusminus}. 
We need to prove that the law of 
$( S_0, S_1, \ldots, S_{T_m^\downarrow-1})$ under $\P_{\pi_+}$ 
equals the law of 
$ (-S_{T_m - 1} - d, -S_{T_m - 2} -d, \ldots, -S_0 - d)$ under $\P_{\pi_-}$.
Similarly to the proof of~\eqref{eq: time reversal},
by the duality principle for random walks this reduces to the equality  
$$\P_{\pi_+}((S_0, S_k) \in \cdot, S_{k+1}<0)=\P_{\pi_-}(R_1(S_0,S_k) \in \cdot, S_{k+1} \ge 0)$$
of measures on $\tilde{\ZZ}_d^+ \times \tilde{\ZZ}_d^+ $ for $k\in \N_0$. Use the definitions of $\pi_+$, $\pi_-$, and $R_1$ to write this as
\begin{align*}
&\mathrel{\phantom{=}} \P_{\lambda_d}((S_0, S_k) \in d x_0 \otimes d x_k) \P(X_1>x_0)\P(X_1 < -x_k) \\
&=  \P_{\lambda_d}(R_1(S_0, S_k) \in d x_0 \otimes d x_k) \P(X_1\geq x_0+d)  \P(X_1 \le - x_k -d).
\end{align*}
This equality holds by \eqref{eq: RW reversal} and the fact that $\P(X_1>x) = \P(X_1\geq x+d)$ for $\lambda_d$-a.e.~$x$. 
\end{proof}

\section{Convergence to the stationary distribution} \label{sec: Convergence}
For the rest of the paper we assume~\eqref{eq:Assumption} and  investigate convergence in total variation of the law of $O_n$ to the probability distribution $\pi_+$ as $n\to\infty$. 

In the non-arithmetic case convergence in the total variation norm requires additional assumptions on the law of the increments of $S$. We say that the distribution of the increment $X_1$ is  {\it spread out}  if $\P_0(S_k \in \cdot)$ is non-singular with respect to the Lebesgue measure  for some $k \ge 1$. It is clear that this assumption is necessary for the total variation convergence to $\pi_+$ of the law of the chain $O_n$ starting from a point. In fact, if this assumption is violated in the non-arithmetic case, then $\| \P_x( O_n \in \cdot) - \pi_+(\cdot)\|_{\text{TV}} = 1$ for every $x \in \R$ and $n \ge 1$ since $\pi_+$ has density. In this sections we will show that that the spread out assumption is actually sufficient for the total variation convergence. Let us mention that spread out distributions arise often in the context of renewal theory, see Asmussen~\cite[Section~VII]{Asmussen}.


\begin{theorem}
\label{thm: convergence general}
Assume~\eqref{eq:Assumption} and that the distribution of $X_1$ is either arithmetic or spread out. Then 
$$\lim_{n \to \infty}\| \P_x( O_n \in \cdot) - \pi_+(\cdot)\|_{\text{TV}}  = 0 \quad\text{for all $x \in \ZZ_d$.}$$
\end{theorem}

A standard application of the dominated convergence theorem yields another proof of the fact 
(given in full generality by~\cite[Theorem~4]{MijatovicVysotskyMC}) that, 
under the assumptions of Theorem~\ref{thm: convergence general},
$\pi_+$ is the unique stationary distribution of the chain
$(O_n)_{n \ge 0}$ in the class of all probability laws on $\ZZ_d$, including the ones singular with respect to $\pi_+$. 

The convergence in Theorem~\ref{thm: convergence general} may fail for every
starting point $x\in\ZZ_0$ in the case of general non-arithmetic distributions
of increments, e.g.\ for discrete non-arithmetic distributions, but $\pi_+$
remains the unique stationary distribution of $O$ by~\cite[Corollary to Theorem~4]{MijatovicVysotskyMC}.  Therefore one may argue that the total variation
metric is too fine for the study of convergence of the chain of overshoots for
general zero
mean random walks.  It is feasible that the convergence 
holds in other metrics under less restrictive assumptions than those in
Theorem~\ref{thm: convergence general} but we did not succeed in proving
results of such type; see the discussion in
Section~\ref{sec:Concluding_remarks} below. 

It is well known that 
under the spread out assumption on the increments of a random walk, a successful coupling of the walks 
started at arbitrary distinct points $x,y \in \ZZ_0$ can be defined, implying in particular
$\lim_{n \to \infty}\| \P_x( S_n \in \cdot) -\P_y( S_n \in \cdot)\|_{\text{TV}}  = 0$, 
see e.g.\ Theorem~6.1 of Chapter~3 in Thorisson~\cite{Thorisson}.
However, this coupling yields only a shift-coupling~\cite[Section~3.1]{Thorisson} of the chains of overshoots started at $x$ and $y$. 
Thus only the Cesaro total variation convergence~\cite[Section~3.2]{Thorisson} of $O$ can be deduced from these results, which 
is weaker than the convergence stated in Theorem~\ref{thm: convergence general}.  
Our proof  of Theorem~\ref{thm: convergence general}
rests on the crucial property of the Markov chain 
$(O_n)_{n \ge 0}$ stated below in Proposition~\ref{prop: minorization}, 
implying that a successful coupling of the chains of overshoots started at any distinct levels can be constructed for any 
span $d\in[0,\infty)$.
We do not exhibit the coupling construction in this paper but instead apply 
Theorem 4 in Roberts and Rosenthal~\cite{RobertsRosenthal2004}, which is established using this coupling.

For any measure $\mu$ on $\ZZ_d$, denote respectively by $\mu^a$ and $\mu^s$ its absolutely continuous and singular components with respect to $\lambda_d$. We will slightly abuse this notation for distributions of random variables and write, say, $\P^a_x( O_1 \in \cdot)$ instead of $(\P_x( O_1 \in \cdot))^a$. 
We reserve the term ``density'' to mean the density with respect to the Lebesgue measure $\lambda_0$ without referring to the measure.
The set $\mathcal{X}_+:=[0,M_+) \cap \ZZ_d$, where $M_+ := \sup( \supp(X_1))$, is
the actual state space of the Markov chain of overshoots:
for any $x\in\ZZ_d$ and $n \in \N$ we have
$\P_x(O_n \in \mathcal{X}_+)=1$. 
Moreover, the equality  $\pi_+(\mathcal{X}_+)=1$ holds true. 
\begin{proposition}
\label{prop: minorization}
Assume~\eqref{eq:Assumption} and that the distribution of $X_1$ is either arithmetic or spread out. 
Then the measures $\P_x^a( O_1 \in \cdot)$ and 
$\pi_+(\cdot)$ are equivalent for any $x\in\ZZ_d$. 
Put differently,
for any $x \in \ZZ_d$
there exists a version of the density $\frac{d}{d \lambda_d} \P^a_x( O_1 \in dy)$ that is strictly positive for all $y \in \mathcal{X}_+$.
\end{proposition}

\begin{proof}[\bf Proof of Theorem~\ref{thm: convergence general}]
Proposition~\ref{prop: minorization} implies that with positive probability, the chain of overshoots visits in a single step
any Borel set $A\subseteq\ZZ_d$ satisfying $\pi_+(A)>0$.  
This means that the Markov chain $(O_n)_{n \ge 0}$ is
$\pi_+$-irreducible and aperiodic in the sense of Meyn and Tweedie~\cite[Sections 4.2 and 5.4]{MeynTweedie}. 
By Theorem~\ref{thm: stationary distr} above, 
$(O_n)_{n \ge 0}$
has a stationary distribution $\pi_+$.
Then  Theorem 4 in Roberts and Rosenthal~\cite{RobertsRosenthal2004}, which
applies to $\psi$-irreducible aperiodic  Markov chains with
a stationary distribution on a general state space with a countably generated $\sigma$-algebra, implies the total variation convergence in
Theorem~\ref{thm: convergence general} for $\pi_+$-a.e. $x \in \ZZ_d$. 

Since $\P_x(O_1 \in \mathcal{X}_+)=1$ for every  $x \in \ZZ_d$, we will
conclude the proof of Theorem~\ref{thm: convergence general} if we show that
the non-convergence set $N:=\{x \in \mathcal{X}_+: \limsup_{n \to \infty}
\| \P_x( O_n \in \cdot) - \pi_+(\cdot)\|_{\text{TV}}  > 0 \}$ is empty.  In the
arithmetic case ($d>0$) this is clear by the fact that every point of
$\mathcal{X}_+$ has positive $\pi_+$-measure and $\pi_+(N)=0$. In the
non-arithmetic case ($d=0$)  first note that since 
the Borel $\sigma$-algebra on $\mathcal{X}_+$ is countably generated,
the function 
$x\mapsto \| \P_x( O_n \in \cdot) - \pi_+(\cdot)\|_{\text{TV}}$ 
is measurable for every $n\in\N$ by Roberts and Rosenthal~\cite[Appendix]{RobertsRosenthal97}, making the set $N$ measurable.
Thus the claim will follow by a standard application of the
strong Markov property and the dominated convergence theorem if we show that
the chain $(O_n)_{n \ge 0}$ hits the convergence set $\mathcal{X}_+ \setminus
N$ with probability one when started in $N$. Put differently, we need to prove that $\P_x(O_n \in
N, \forall n \in \N)=0$ for every $x \in N$. 

Since $\pi_+(N)=\lambda_0(N)=0$ 
we have 
$\P_x(S_m \in N) =\P_x^s(S_m \in N)$ for all $m\in\N$.
Hence, 
\begin{align} \label{eq:N set bounds}
&\mathrel{\phantom{=}}  \P_x(O_n \in N, \forall n \in \N) \le \liminf_{n \to \infty} \P_x(O_n \in N) \notag \\
&\le \liminf_{n \to \infty} \sum_{m=2n}^\infty \P_x(S_m \in N) 
\le \liminf_{n \to \infty} \sum_{m=2n}^\infty \P_x^s(S_m \in \R),
\end{align}
where in the second inequality we used the identity $O_n = S_{T_n}$ and the
fact that $T_n \ge 2n$ for $x \ge 0$, cf.\ \eqref{eq:crossing_time_def} and~\eqref{eq:Chain_definitions}. 
By the definition of spread out distributions, we have
$\P_x^s(S_k \in \R) = \P^s(S_k' \in \R)<1 $ for some $k \ge 1$. Then, using that the convolution of an absolutely continuous measure with any other measure is absolutely continuous, we get
$$
\P^s(S_m' \in \R) = \left( \bigl(\P^s(S_k' \in \cdot) + \P^a(S_k' \in \cdot) \bigr)^{*\lfloor m/k \rfloor} * \P(S'_{m-k\lfloor m/k \rfloor} \in \cdot) \right)^s(\R)
\le \bigl(\P^s(S_k' \in \R)\bigr)^{\lfloor m/k \rfloor}
$$ 
for any integer  $m \ge 1$, where
$\lfloor c\rfloor$ denotes the largest non-negative integer smaller or equal to a $c \ge 0$. Hence the sequence $\P_x^s(S_m \in \R)$, which equals $\P^s(S_m' \in \R)$, decays exponentially fast to zero as $m \to \infty$, and it follows that the last bound in \eqref{eq:N set bounds} is zero.
\end{proof}

\begin{proof}[\bf Proof of Proposition~\ref{prop: minorization}]
Pick any $x \in \ZZ_d$ and and denote by $y$ an arbitrary element in $\mathcal{X}_+$. Consider two cases.

\underline{Arithmetic distributions.} We need to prove that $\P_x(O_1 = y)>0$.

Since $y < M_+$, there exists a $z \in \ZZ_d^+$ such that $z >y$ and $\P(X_1 = z)>0$. Further, according to the definition of $\ZZ_d$, there exists an integer $k \ge 1$ such that $\P_x(S_k = y-z)>0$; see, e.g., Spitzer~\cite[Propositions 2.1 and 2.5]{Spitzer}. 
Then 
\begin{equation}
\label{eq:initial_inequality}
\P_x(O_1 = y) \ge \P_x(S_k = y-z, T_1 > k) \cdot \P(X_1 = z),
\end{equation}
and it remains to show that the first factor in the r.h.s.\ is positive.

Denote by $\Sym (k)$ the symmetric group on the set $\{1,\ldots, k\}$.
For any permutation $\sigma \in \Sym(k)$, define a new random walk
$S(\sigma)=(S_n(\sigma))_{n \ge 0}$ by $S_n(\sigma):= S_0 + X_{\sigma(1)} + \ldots + X_{\sigma(n)}$ 
for $1 \le n \le k$ and $S_n(\sigma) := S_n$ for $n \ge
k$. Denote by $T_1(\sigma)$ the first up-crossing time of the 
level zero by $S(\sigma)$ (cf.\ \eqref{eq:crossing_time_def}), and let $\xi$ be the
number of negative terms among $X_1, \ldots, X_k$. 

Note that 
on the event 
$A_\sigma:=\{\xi\ge 1,X_{\sigma(1)} <0, \ldots, X_{\sigma(\xi)} <0\}\cup\{\xi=0\}$
the sequence $(S_n(\sigma))_{n\in\{\sigma(\xi), \ldots, k\}}$ is non-decreasing
(on $\{\xi=0\}$, we define $\sigma(\xi):=0$).  
Then, since $y-z<0$
and $S_k=S_k(\sigma)$, 
we have
\begin{equation}
\label{eq:Main_inclusion}
\{S_k = y-z\}\cap A_\sigma \subset \{S_k(\sigma) = y-z,T_1(\sigma)>k\}.
\end{equation}
Recall that the cardinality of $\Sym (k)$ is $k!$ 
and note that
\begin{equation}
\frac{1}{k!}\sum_{\sigma \in \Sym(k)}\I(A_\sigma)=
\I(\xi=0) + \I(\xi>0)\xi!(k-\xi)!/k!=
1/{k \choose \xi }\geq 1/ {k \choose \lfloor k/2 \rfloor}.
\label{eq:A_sigma}
\end{equation}
Since the laws of the random walks $S$ an $S(\sigma)$ coincide for all $\sigma \in \Sym(k)$, we get 
\begin{align}
\nonumber
\P_x(S_k = y-z, T_1 > k) &= \frac{1}{k!} \E_x \sum_{\sigma \in \Sym(k)} \I(S_k(\sigma) = y-z, T_1(\sigma) > k) \\
\nonumber
&\ge \frac{1}{k!} \E_x \I(S_k = y-z)\sum_{\sigma \in \Sym(k)}\I(A_\sigma)\\ 
& \geq \P_x(S_k = y-z)/{k \choose \lfloor k/2 \rfloor} >0,
\label{eq: positive permute}
\end{align}
where the first inequality holds by~\eqref{eq:Main_inclusion} and the second by~\eqref{eq:A_sigma}.
Combined with~\eqref{eq:initial_inequality}, this proves 
$\P_x(O_1 = y)>0$,
and hence the proposition holds in the arithmetic case.

\underline{Spread out distributions.} 
%
We say that measures $\mu$ and $\nu$ on $\ZZ_0=\R$ satisfy
$$
\mu(du) \geq \nu(du) \text{ on an interval } I\subset \R
$$
if $\mu(B)\geq \nu(B)$ for any Borel set $B\subset I$.
Note that $\mu(du) \geq c \lambda_0(du)$ on $I$
implies $\mu^a(du)\geq c \lambda_0(du)$ on $I$. 
In this case there exists a version of the  density of $\mu^a$ which is  bounded from below on $I$
by the positive constant $c$. 

Since the distribution of $X_1$ is spread out,
there exist $\varepsilon_1, h > 0$, an integer $k_1 \ge 1$, and a real $a$ such that 
$\P_0(S_{k_1} \in du)\geq \varepsilon_1 \lambda_0(du)$ on $[a,a+2h]$; see the proof of Proposition~5.3.1 in Meyn and
Tweedie~\cite{MeynTweedie}.   
By the Chung--Fuchs theorem, the zero mean random walk $S$ is topologically
recurrent. Hence for any $b\in\R$ (to be specified later) there exists 
$k_2 = k_2(b-x) \in\N$ such that 
\begin{equation}
\label{eq: local estimate}
\varepsilon_2 = \varepsilon_2(b-x):=\P_x \bigl(S_{k_2} \in [b - a -h , b - a  ] \bigr) >0.
\end{equation}
Let $k=k(b-x):=k_2(b-x)+k_1$. 
Then, for any 
$v\in[b-a-h, b-a]$,
we have 
$\P_0( S_{k_1}\in du - v)\geq \varepsilon_1 \lambda_0(du)$ on $[b,b+h]$,
since 
$u-v\in[a,a+2h]$ and 
the Lebesgue measure 
$\lambda_0$ is invariant under translations. 
Hence on the interval 
$[b,b+h]$
the following holds:
\begin{multline*}
\P_x(S_k\in du ) \geq \P_x(b-a-h\leq S_{k_2}\leq b-a, S_k\in du ) \\
= \int_{[b-a-h, b-a]} \P_x(S_{k_2}\in dv) \P_0( S_{k_1}\in du - v) \geq \varepsilon_1 \varepsilon_2 \lambda_0(du).
\end{multline*}
In particular, 
the density of $\P_x^a(S_k \in \cdot)$ is bounded below by $\varepsilon_1 \varepsilon_2$ on $[b,b+h]$.

Since $y<M_+$, we can choose  $z>y$ such that $\varepsilon_3= \varepsilon_3(y):=\P(X_1 \in [z, z+h/2]) >0$. 
Set $b':=y-z-3h/4$ and $h':=\min(h, y-z-3h/4)$, and let 
$k':=k_2(b'-x)+k_1$ 
and
$\varepsilon_2':=\varepsilon_2(b'-x)$
as in~\eqref{eq: local estimate}.
Then
$\P_x(S_{k'} \in du) \ge \varepsilon_1 \varepsilon_2' \lambda_0(du)$ on $[b',b'+h]$. 
Moreover, 
substituting the events $\{S_k=y-z\}$ and $\{S_k(\sigma)=y-z\}$, where $\sigma\in\Sym{(k)}$, in the proof of the arithmetic case above
by
$\{S_{k'}\in[b',b'+h')\}$
and $\{S_{k'}(\sigma)\in[b',b'+h')\}$, where $\sigma\in\Sym{(k')}$ and $b'+h' \le 0$, 
yields a bound analogous to~\eqref{eq: positive permute}: 
\begin{equation}
\label{eq:lower_bound_with_T_1}
\P_x(S_{k'} \in du , T_1 > k')\ge \frac{\varepsilon_1 \varepsilon_2' }{{k'
\choose \lfloor k'/2 \rfloor} } \lambda_0(du) \quad \text{on }[b', b'+h').
\end{equation}
On the event 
$\{S_{k'} \in [b',b'+h') , T_1 > k', X_{k'+1} \ge z \}$ 
we have 
$O_1=S_{k'}+X_{k'+1}$.
The Markov property at $k'$ implies that on the interval $[y-h/4 ,\min(z, y+h/4))$  we have 
\begin{multline}
\label{eq:lower_bound_trans_prob}
\P_x(O_1 \in dv) \ge \int_{[z,z+h/2]} \P_x(S_{k'}\in dv-u, T_1>k') \P(X_{k'+1} \in du)  \\
\ge \frac{\varepsilon_1 \varepsilon_2' }{{k' \choose \lfloor k'/2 \rfloor} }  \P(X_1 \in [z,z+h/2]) \lambda_0(dv). 
\end{multline}
The second inequality holds by~\eqref{eq:lower_bound_with_T_1}
and the translation invariance of $\lambda_0$
since,
for 
$v\in[y-h/4,\min(z, y+h/4))$ 
and
$u\in[z,z+h/2]$, we have 
$v-u\in[b',b'+h')$.
Since we can partition $\mathcal{X}_+$ by a countable subcollection of the intervals
$\{[y-h/4,\min(z, y+h/4)):y\in\mathcal{X}_+\}$,
by~\eqref{eq:lower_bound_trans_prob}
there exists a version of 
the density $p(x, \cdot)$ of
$\P_x^a(O_1 \in \cdot)$ satisfying 
\begin{equation}
\label{eq: density >}
p(x, y) \ge \frac{\varepsilon_1 \varepsilon_2' \varepsilon_3 }{ {k' \choose \lfloor k'/2 \rfloor} }>0 \qquad 
\text{for all $y \in \mathcal{X}_+$.}
\end{equation}
\end{proof}

\section{Rate of convergence to the stationary distribution} \label{Sec: convergence rate}
In this section we present results on the rate of convergence in
Theorem~\ref{thm: convergence general}. We will use the following norm: for any
function $f: \ZZ_d^+ \to [1, \infty)$, the {\it $f$-norm} of a signed measure
$\mu$ on $ \ZZ_d^+$ is $$\| \mu \|_f:= \sup_{g: |g| \le f} \int_{ \ZZ_d^+} g(x)
\mu (d x).$$ 
In particular, for $f\equiv1$ the following relationship with the total variation norm holds: 
$\|\mu\|_f=2\|\mu\|_{\text{TV}}$.
Clearly, convergence in any $f$-norm is stronger than the total variation convergence. We will only need the $V_\gamma$-norms, where $V_\gamma(x):=1+x^\gamma$ with $\gamma \ge 0$.

Further, define the set of bivariate parameters
$$\mathcal{I}:=\{ (\alpha, \beta): 1< \alpha < 2, |\beta| < 1 \}.$$ 
For a random variable $X$, we write $X \in \mathcal{D}(\alpha, \beta)$ for a pair $(\alpha, \beta) \in \mathcal{I}$ if the distribution of $X$ belongs to the domain of attraction of a strictly stable law with the characteristic function 
$$\chi_{\alpha, \beta}(t)=\exp \bigl( -c |t|^\alpha (1- i \beta \sgn(t) \tan (\pi \alpha /2 ))\bigr)$$ for some $c>0$.
Denote $\mathcal{D}:=\cup_{(\alpha, \beta) \in \mathcal{I} }\mathcal{D}(\alpha, \beta)$.
The quantity 
$$p:= 1/2 + (\pi \alpha)^{-1} \arctan(\beta\tan(\pi \alpha/2)),$$ 
which is called the {\it positivity parameter} of the stable law, ranges over the open interval $(1- 1/\alpha, 1/\alpha)$ on $\mathcal{I}$,
see in Bertoin~\cite[Section 8.1]{Bertoin}. 

\begin{theorem}
\label{thm: geometric ergodicity}
Assume~\eqref{eq:Assumption} and that the distribution of $X_1$ is either
arithmetic or spread out. In addition, assume either $\E X_1^2 <\infty$
with $\gamma \in \{0,1\}$ or $X_1 \in \mathcal{D}$ with $\gamma \in (0,
\min\{\alpha p, \alpha(1-p)\} )$. Then there exist constants $r \in (0,1)$ and
$c_1>0$ such that \begin{equation} \label{eq: geometric ergodicity}
\| \P_x(O_n \in \cdot ) - \pi_+ (\cdot)\|_{V_\gamma} \le c_1 (1+x^\gamma) r^n, \qquad x \in  \ZZ_d^+.
\end{equation}
\end{theorem}

Equation \eqref{eq: geometric ergodicity} with $\gamma=0$ translates to a uniform
(in $x$) convergence at a geometric rate in the total variation norm. Thus the
chain of overshoots is {\it uniformly ergodic} if the increments have finite
variance; see Meyn and Tweedie~\cite[Theorem 16.0.2]{MeynTweedie}.

Our proof of Theorem~\ref{thm: geometric ergodicity} rests on  two statements. The first can be viewed as a uniform version of Proposition~\ref{prop: minorization} stated in a slightly different form  to avoid measurability issues.

\begin{proposition}
\label{prop: uniform}
Under the assumptions of Theorem~\ref{thm: geometric ergodicity}, for any $K>0$
in the case $X_1 \in \mathcal{D}$ and for $K = \infty$ in the case $\E X_1^2
<\infty$, 
there exists a measurable function $g_K: \mathcal{X}_+ \to (0,
\infty)$ such that 
\begin{equation} \label{eq: uniform probability}
\P_x(O_1\in B)  \ge \int_B g_K(y)\lambda_d(dy)  \qquad \text{for all } x \in  \ZZ_d^+\cap[0,K) \text{ and Borel sets } B\subset \mathcal{X}_+.
\end{equation}
\end{proposition}

\begin{remark}
Note that~\eqref{eq: uniform probability} implies 
$\P_x(O_1\in B)  \ge \int_B g_K(y)\lambda_d(dy)>0$  
for any Borel set $B$ with $\lambda_d(B)>0$.
In particular,
every compact set $C \subset \ZZ_d^+$ with non-empty interior in $\ZZ_d^+$ (or the
whole set  $\ZZ_d^+$ in the finite variance case) is {\it small} with respect
to the measure $g_{\text{diam}(C)}(y) \I_{\mathcal{X}_+} (y) \lambda_d(dy)$;
see Meyn and Tweedie~\cite[Section 5.2]{MeynTweedie} for the definition of
small sets. The proposition also yields that the Markov chain $(O_n)_{n\ge0}$ is 
{\it strongly aperiodic} and satisfies the {\it minorization condition}, cf.\ Sections 5.4 and
5.1 in~\cite{MeynTweedie}, respectively.  
\end{remark}

\begin{remark}
Our proof, based on Stone's local limit theorem, 
actually implies that the inequality in~\eqref{eq: uniform probability}
with finite $K$ is also valid for
asymptotically stable distributions of increments with $1<\alpha<2, |\beta | =1$ and with $\alpha =2$. 
Moreover, it is plausible that \eqref{eq: uniform probability} holds under assumptions 
of Proposition~\ref{prop: minorization}, i.e.\ without any assumptions on the tail behaviour of $X_1$ beyond~\eqref{eq:Assumption}.
\end{remark}

Second, we need the following {\it geometric drift condition}. We will prove it using results of renewal theory. 

\begin{proposition}
\label{prop: drift}
Under the assumptions of Theorem~\ref{thm: geometric ergodicity}, there exist constants $\rho \in (0,1)$ and $L>0$ such that
\begin{equation}
\label{eq: drift}
\E_x O_1^\gamma \le \rho x^\gamma +L, \qquad x \in \ZZ_d^+.
\end{equation}
\end{proposition}

Put together, Propositions~\ref{prop: uniform} and \ref{prop: drift} imply
Theorem~\ref{thm: geometric ergodicity} 
via 
Theorems 15.0.1 and 16.0.2 and Proposition 5.5.3 
in
Meyn and Tweedie~\cite{MeynTweedie}.

\begin{proof}[\bf Proof of Proposition~\ref{prop: uniform}] 

\underline{The case $X_1 \in \mathcal{D}(\alpha, \beta)$.}
In the arithmetic case 
the set $\ZZ_d^+\cap[0,K)$
has a finite number of elements and the claim 
follows from Proposition~\ref{prop: minorization}.
For spread out
distributions we have $d=0$, implying 
$\ZZ_0^+\cap[0,K)=[0,K)$,
and it is clearly sufficient to prove that there exist a measurable function $g_K: \mathcal{X}_+ \to (0, \infty)$ and for every $x$, a version $p(x, y)$ of the density of $\P^a_x( O_1 \in dy)$ such that 
\begin{equation} \label{eq: uniform density}
\inf_{x \in  [0,K)} p(x,y) \ge g_K(y)>0 \qquad \text{for all $y \in \mathcal{X}_+$.}
\end{equation}
We will do this by
refining the argument in the proof of
Proposition~\ref{prop: minorization}. 

Pick $y \in \mathcal{X}_+$ and consider the estimate in~\eqref{eq: density >}. Note
that $\varepsilon_1$ does not depend on $x$ and $y$ while $\varepsilon_3$
depends only on $y$ through the choice of $z > y$. 
By decomposing $\mathcal{X}_+$ into a pair-wise disjoint collection of countably many bounded half-open intervals
and choosing the same $z$ for all $y$ in each of the intervals makes 
$y\mapsto\varepsilon_3(y) = \P(X_1 \in [z, z+h/2])$ a measurable function of $y$. 
Therefore it suffices to check that $\varepsilon_2'$ can be bounded away from
zero and $k'$ can be bounded from above, both uniformly in $x\in [0,K)$ and $y$ in each of the 
intervals in the partition of $\mathcal{X}_+$. 
These claims will follow once we establish a refined
version of~\eqref{eq: local estimate}: for any compact interval $I$ in $\R$ and $h>0$, there exists
an integer $m \ge 1$ such that
\begin{equation}
\label{eq: local estimate 2}
\inf_{x\in[0,K), u\in I} \P_x (S_m \in [u , u + h ] ) >0.
\end{equation}

Possibly the easiest way to prove~\eqref{eq: local estimate 2} is to apply
Stone's local limit theorem which holds for non-lattice asymptotically stable
distributions \cite[Corollary~1]{Stone1965}: if the sequence $(b_n)_{n \ge 1}$
tending to infinity is such that $S_n/b_n$ converges weakly 
to a strictly stable law with the characteristic function $\chi_{\alpha, \beta}$ given above, then 
$$
\P_x (S_n \in [ u , u + h ) ) = \P_0 (S_n \in [ u-x , u-x + h ) )  = (hp_{\alpha, \beta}(0)  + o(1))b_n^{-1}
$$
uniformly in 
$x \in [0, K)$ and $u\in I$ as $n \to \infty$, where $p_{\alpha, \beta}$, the density of the stable
law defined by $\chi_{\alpha, \beta}$, is strictly positive and continuous at $0$ for $(\alpha, \beta) \in \mathcal{I}$. Hence the inequality in~\eqref{eq: local estimate 2} holds for all $n$
sufficiently large. 

\underline{The case $\E X_1^2 < \infty$.} Note that in this case the above proof 
implies~\eqref{eq: uniform probability} for any finite $K>0$.
In order to construct $g_\infty:(0,\infty)\to\mathcal{X}_+$,  
let $T_{(-\infty,L)}:=\min\{n\geq 0:S_n<L\}$ be the moment of the first
entrance of the walk $(S_n)_{n \ge 0}$ to the half-line $(-\infty, L)$,  where
$L:=d+1>0$. 
For any Borel set $B$ in $\mathcal{X}_+$  we have
\begin{multline*} 
\P_x(O_1\in B)=\int_{(-\infty,L)} \P_z(O_1\in B)\P_x(S_{T_{(-\infty,L)}} \in dz)
\ge  \int_{(0,L)} \P_z(O_1\in B)\P_x(S_{T_{(-\infty,L)}} \in dz)\\
\ge
\P_x(S_{T_{(-\infty,L)}} \in (0,L))\int_B g_L(y)\lambda_d(dy) 
=
 \P_{x-L}(O_1^\downarrow \in (-L, 0))\int_B g_L(y)\lambda_d(dy), 
\end{multline*}
where $g_L$ is the lower bound in~\eqref{eq: uniform density} that corresponds to the interval $(0,L)$
and the equality 
$\P_x(S_{T_{(-\infty,L)}} \in (0,L))= \P_{x-L}(O_1^\downarrow \in (-L, 0))$ 
holds by the definition of $O_1^\downarrow$ in~\eqref{eq: chains downcrossing}.
By~\eqref{eq: inf down}, under the assumption $\E X_1^2 < \infty$,
$\P_x(O_1^\downarrow \in \cdot)$ converges weakly as $x \to \infty$ to a 
distribution which assigns positive mass to $(-L, 0)$. Hence
there exist constants $c_0,K_0>0$ such that 
$\P_x(O_1\in B)\geq c_0 \int_B g_L(y)\lambda_d(dy)$ for all $x\geq K_0$ and all Borel sets $B$ in $\mathcal{X}_+$.
The positive function $g_\infty:=\min\{g_{K_0},c_0g_L\}$  satisfies the inequality in~\eqref{eq: uniform probability} for $K=\infty$.
This concludes the proof of the proposition. 
\end{proof}

\begin{proof}[\bf Proof of Proposition~\ref{prop: drift}]
We will require a representation of the overshoots
$O_1$ and $O_1^\downarrow$ as residual lifetimes of renewal processes of ladder
heights. The sequence of descending ladder heights $(H_k^-)_{k \ge 0}$ of the
random walk $S'$ (recall that $S'_0=0$) satisfies $H_0^-=0$ and its increments
$Y_k:=H_k^- - H_{k-1}^-$ are negative i.i.d.\ random variables distributed as
$H_1^-$, the first strictly negative value of $S'$. For any $x \ge 0$, denote by $R^-(x):= \sup\{H_k^- + x: k \ge 1, H_k^- < -x \}$ the overshoot at down-crossing of the level $-x$. Then, by definition~\eqref{eq: chains downcrossing}, we have
\begin{equation} \label{eq:residual time form}
O_1^\downarrow  = R^-(S_0) \, \text{ on } \{S_0 \ge 0 \}.
\end{equation}
In particular, this implies \eqref{eq: inf down} under the assumption $\E X_1^2<\infty$.

Clearly, there is a similar representation for the overshoot $O_1$ at the first
up-crossing: $$O_1 = \tilde{R}^+(-S_0)  \, \text{ on } \{S_0 < 0 \},$$
where $\tilde{R}^+(x):= \inf\{H_k^+ -x: k \ge 1, H_k^+ \ge x \}$ is the
non-negative residual lifetime at time $x > 0$ for the ascending ladder height
process $(H_k^+)_{k \ge 1}$ of the random walk $S'$. The increments of this
process are i.i.d.\  and have the same common distribution as $H_1^+$, the first
strictly positive value of $S'$.

\underline{The case $X_1 \in \mathcal{D}(\alpha, \beta)$.} We need to estimate
$\E_x O_1^\gamma$ and we start with the following bounds. For any $x>0$, 
denote
$T(x):= \inf \{k \ge 1: |H_k^-| > x\} $ and
$T'(x):= \inf \{k \ge 1: |Y_k| > x\}$ with the convention $\inf_\varnothing:=\infty$. 
By the assumption  $|\beta| <1$ the distribution of $|Y_1|$ 
has unbounded support, implying $T'(x)<\infty$ a.s.\ for any real $x$.
Since $|R^-(x)| < |Y_{T(x)}| \le |Y_{T'(x)}|$ a.s., we have $$\E_x |O_1^\downarrow|^\gamma = \E |R^-(x)|^\gamma
< \E |Y_{T(x)}|^\gamma \le \E |Y_{T'(x)}|^\gamma.$$ 
Clearly, a similar estimate applies for $\E_{-x} O_1^\gamma$. 
Since the law of 
$|Y_{T'(x)}|$ 
equals that of $|H_1^-|$ conditioned to be greater than $x$, 
we have
\begin{equation}
\label{eq: moment estimates}
\E_x |O_1^\downarrow|^\gamma \le \frac{\E \bigl [|H_1^-|^\gamma \I_{\{|H_1^-| > x\}} \bigr]}{\P(|H_1^-| > x) } 
\qquad\text{and}\qquad \E_{-x} O_1^\gamma \le \frac{\E \bigl [(H_1^+)^\gamma \I_{\{H_1^+ \ge x\}} \bigr]}{\P(H_1^+ \ge x) }.
\end{equation}
 Note that the r.h.s.'s of these inequalities are monotone in $x$ since $Y_{T'(x)}$ is non-decreasing in $x$ a.s. 


Recall that, since $|\beta|<1$, we have $\alpha p  <1$ and $\alpha q<1$, where $p$ is the positivity parameter 
introduced in the beginning of the section and
$q:=1-p$. By
Theorem~9 of Rogozin~\cite{Rogozin1971} we have $|H_1^-| \in \mathcal{D}( \alpha q,1)$. 
Hence the renewal theorem of
Dynkin~\cite[Theorem~3]{Dynkin1955} applies to the residual lifetime process $(|R^-(x)|)_{x >0}$ and so the distributions $\P_x(-O_1^\downarrow/x \in \cdot)$ converge
weakly as $x \to \infty$ to the distribution with the density $$g_{\alpha q}(t)
= \pi^{-1} \sin(\pi \alpha q) t^{-{\alpha q}} (1+t)^{-1}, \quad t>0,$$ 
supported on the positive half-line. Recalling that $\gamma \in (0, \alpha q)$, we will obtain  that
\begin{equation}
\label{eq: E overshoot down}
\lim_{x \to \infty} \frac{\E_x |O_1^\downarrow|^\gamma}{x^\gamma} = \int_0^\infty t^\gamma g_{\alpha q}(t) dt = \frac{\sin(\pi \alpha q )}{\sin( \pi (\alpha q - \gamma) )}=:c_{\alpha, q}(\gamma)
\end{equation}
once we check the uniform integrability of the distributions $\P_x(|O_1^\downarrow/x|^\gamma \in \cdot )$. 

Consider the numerator in the first estimate in \eqref{eq: moment estimates}. Using Karamata's theorem (see Bingham et al.~\cite[Proposition 1.5.10]{Bingham+}) and the fact that the tail probability $\P(|H_1^-| >x)$ is regularly varying at infinity with index $-\alpha q$, we get
\begin{multline*}
\E \bigl [|H_1^-|^\gamma \I_{\{|H_1^-| > x\}} \bigr]= \int_{x+}^\infty t^\gamma \P(|H_1^-| \in dt) \\
= x^\gamma \P(|H_1^-| > x) + \int_x^\infty \gamma t^{\gamma-1} \P(|H_1^-| > t) dt \sim\frac{\alpha q}{\alpha q -\gamma}  x^\gamma \P(|H_1^-| > x) 
\end{multline*}
as $x \to \infty$. Hence, by \eqref{eq: moment estimates}, we have
$$\limsup_{x \to \infty} \E_x (|O_1^\downarrow|/x)^{\gamma} \le \frac{\alpha q}{\alpha q -\gamma}.$$
Since the above computations work for any $\gamma \in (0, \alpha q )$, the
$\limsup_{x \to \infty} \E_x (|O_1^\downarrow|/x)^{\gamma_0} $ is finite for
any $\gamma_0 \in (\gamma, \alpha q)$. This yields the required uniform
integrability, and \eqref{eq: E overshoot down} follows.

Further, consider $\E_{-x} O_1^\gamma = \E (\tilde{R}^+(x))^\gamma$. The
result~\cite[Theorem~3]{Dynkin1955} used above applies only to (positive)
residual times $R^+(x):= \inf\{H_k^+ -x: k \ge 1, H_k^+ > x \}$, where  $H_1^+
\in \mathcal{D}( \alpha p,1)$ by~\cite[Theorem~9]{Rogozin1971}. However, it
gives the weak convergence  of $R^+(x)/x $ as $x \to \infty$ to the
distribution with density $g_{\alpha p}(t)$. This yields the weak convergence
of 
$\P_{-x} (O_1/x \in \cdot)$ to the same limit since $R^+(x) = \tilde{R}^+(x)$ on the event $\{R^+(x-1) >1\}$ whose probability tends to $1$ as $x \to \infty$. 
Then,  recalling that $\gamma \in (0, \alpha p)$, we obtain the following analogue of~\eqref{eq: E overshoot down}:
\begin{equation}
\label{eq: E overshoot up}
\lim_{y \to \infty} \frac{\E_{-y} O_1^\gamma}{y^\gamma} = \frac{\sin(\pi \alpha p )}{\sin( \pi (\alpha p - \gamma) )} = c_{\alpha, p}(\gamma).
\end{equation}

We now apply the strong Markov property of the random walk $S$ at $T_1^\downarrow$: for any  $R>0$,
\begin{equation}
\label{eq: I_1 + I_2}
\E_x O_1^\gamma = \int_{0+}^R \E_{-y} O_1^\gamma \cdot \P_x (O_1^\downarrow \in -dy )  + \int_{R +}^\infty [\E_{-y} (O_1/y)^\gamma]  y^\gamma \P_x (O_1^\downarrow \in -dy ).
\end{equation}
The first term is bounded uniformly in $x$ for any fixed $R$:
\begin{equation}
\label{eq: I_1}
\int_{0+}^R \E_{-y} O_1^\gamma \cdot \P_x (O_1^\downarrow \in -dy ) \le \sup_{0
< y \le R} \E_{-y} O_1^\gamma \le \frac{\E \bigl [(H_1^+)^\gamma \I_{\{H_1^+
\ge R\}} \bigr]}{\P(H_1^+ \ge R) } \le \frac{\E\bigl[ (H_1^+)^\gamma\bigr] }{\P(H_1^+ \ge
R) } < \infty, 
\end{equation}
where in the second inequality we used the second inequality in \eqref{eq: moment estimates}, whose r.h.s.\ is monotone. For the second term in \eqref{eq: I_1 + I_2}, 
by~\eqref{eq: E overshoot up}, we make the expression $\E_{-y} (O_1/y)^\gamma$ 
in the integrand arbitrarily close to $c_{\alpha, p}(\gamma)$  by taking $R$ sufficiently large. Finally, for any fixed $R$,
$$\limsup_{x \to \infty} \int_{0+}^R y^\gamma  \P_x (O_1^\downarrow \in -dy ) \le \lim_{x \to \infty} R^\gamma  \P_x (-O_1^\downarrow \le R) = 0.$$ 
Hence by~\eqref{eq: I_1 + I_2}  there exists a constant $C_R>0$, such that 
$\big|\E_x O_1^\gamma - c_{\alpha,p}(\gamma)\E_x |O_1^\downarrow|^{\gamma} \big|\leq C_R$ 
for all $x>0$. 
By \eqref{eq: E overshoot down} and \eqref{eq: E overshoot up} we obtain
\begin{equation}
\label{eq: E contraction}
\lim_{x \to \infty} \frac{\E_x (O_1)^\gamma}{x^\gamma} =  \frac{\sin(\pi \alpha q )}{\sin( \pi (\alpha q - \gamma) )} \cdot \frac{\sin(\pi \alpha p )}{\sin( \pi (\alpha p - \gamma) )} =:\rho_0.
\end{equation}

Since $0 < \gamma < 1 < \alpha <2$,  the following 
implies
$\rho_0 <1$: 
\begin{eqnarray*}
\sin(\pi \alpha q ) \sin(\pi \alpha p ) - \sin( \pi (\alpha q - \gamma) ) \sin( \pi (\alpha p - \gamma) ) &=& \frac12 \cos(\pi(\alpha - 2 \gamma)) - \frac12 \cos (\pi \alpha) \\
&=& \sin(\pi \gamma) \sin (\pi(\alpha - \gamma)) <0.
\end{eqnarray*}
Thus the
inequality in~\eqref{eq: drift} holds for any $\rho \in (\rho_0, 1)$
since  $\E_x O_1^\gamma$ is locally bounded 
by~\eqref{eq: I_1 + I_2}, \eqref{eq: I_1},  and the fact that for all $R$ sufficiently large and
any $K>0$, $$\sup_{0 \le x \le K} \int_{R +}^\infty [\E_{-y} (O_1/y)^\gamma]
y^\gamma \P_x (O_1^\downarrow \in -dy ) \le (1+ c_{\alpha, p}(\gamma)) \sup_{0
\le x \le K} \E_x |O_1^\downarrow|^\gamma \le  \frac{(1+ c_{\alpha, p}(\gamma)
) \E |H_1^-|^\gamma }{\P(|H_1^-| > K) },$$
where we used \eqref{eq: E overshoot up} for the first inequality and \eqref{eq: moment estimates} for the second one as we did in \eqref{eq: I_1}.



\underline{The case $\E X_1^2 < \infty$.} The case $\gamma =0$ is trivial so take $\gamma=1$.  It is well known that the ladder heights of random walks with finite variance of increments are integrable; see Feller~\cite[Sections XVIII.4 and 5]{Feller}. Moreover, we have the following versions of \eqref{eq: E overshoot down} and \eqref{eq: E overshoot up}:
$$\lim_{x \to \infty} \frac{\E_x |O_1^\downarrow|}{x} = \lim_{y \to \infty} \frac{\E_{-y} O_1}{y} = 0,$$
see Gut~\cite[Theorem 3.10.2]{Gut}. The rest of the proof is exactly as in the first case: by \eqref{eq: I_1 + I_2}, the value of the l.h.s.\ of \eqref{eq: E contraction} is now zero and $\E_x O_1$ is locally bounded. 

\end{proof}

\section{Concluding remarks}
\label{sec:Concluding_remarks}

\subsection{The entrance chain into an interval} \label{sec: entrance interval}
The methods of this paper developed for establishing convergence the chain $O$ of overshoots above zero work without any changes for the Markov {\it chain of entrances} into the interval $[0, h]$ for any $h>0$, defined analogously to~$O$ (cf.~\eqref{eq:crossing_time_def} and \eqref{eq:Chain_definitions}): put $O_n^{(h)}:=S_{T_n^{(h)}}$ for $n \in \N_0$, where
$$
T_0^{(h)}:=0, \quad T_{n}^{(h)}:= \inf\{k>T_{n-1}^{(h)}: S_{k-1} \not \in [0, h], S_k \in [0, h] \}, \qquad n \in\N.
$$
By~\cite[Theorem 3]{MijatovicVysotskyMC}, 
\begin{equation}
\label{eq:pi_h_explicit}
\pi_h:=c_h \I_{[0,h]}(x) (1-\P(x- h \le X_1 \le x)) \lambda_d(dx), \qquad x \in \ZZ_d,
\end{equation}
where $c_h>0$ is a normalizing constant, is the unique stationary distribution of the chain $O^{(h)}$ on $\ZZ_d$. The assertions of Theorems~\ref{thm: convergence general} and~\ref{thm: geometric ergodicity} remain valid if we replace $O_n$ and $\pi_+$ respectively by $O_n^{(h)}$ and $\pi_h$, with $\gamma=0$ in Theorem~\ref{thm: geometric ergodicity}. 

To see this, recall that the proof of Theorem~\ref{thm: convergence general} was based on Proposition~\ref{prop: minorization} describing $\P_x(O_1 \in \cdot)$, which was actually used only for starting points $x$ in $\XX_+=\supp (\pi_+)$, where $\XX_+=[0, M_+) \cap \ZZ_d$ with $M_+=\sup(\supp X_1)$. For the chain $O^{(h)}$, we need to consider only $x \in \supp(\pi_h)$, where $\supp(\pi_h)=([0, M_+) \cup (h+M_-, h]) \cap [0, h] \cap \ZZ_d$ with $M_-:=\inf(\supp X_1)$. 
The case $x \in [0, M_+) \cap [0, h] \cap \ZZ_d$ (which gives the claim if $M_+ \ge h+d$) is actually covered in the proof of the proposition, where we can replace throughout $O_1$ by $O_1^{(h)}$ without any other changes. The remaining case $x \in (h+M_-, h] \cap [0, h] \cap \ZZ_d$ follows by considering the random walk $-S$. Finally, Theorem~\ref{thm: geometric ergodicity} immediately follows from Proposition~\ref{prop: uniform}, which is simply the uniform version of  Proposition~\ref{prop: minorization}, and Proposition~\ref{prop: drift}, which trivially holds with $L=h$.

\subsection{Convergence of the chain of overshoots under minimal assumptions}
By~\cite[Corollary to Theorem~4]{MijatovicVysotskyMC}, the probability law $\pi_+$ is the unique stationary distribution for 
the chain of overshoots $O$ of any random walk satisfying~\eqref{eq:Assumption}. By Theorem~\ref{thm: convergence general}, the laws of $O_n$ converge to $\pi_+$ in the total variation distance for random walks with either arithmetic or spread out distributions of increments. Our intuition coming from renewal theory
suggests that the following 
hypothesis is plausible.  
\begin{conj}
Under assumption \eqref{eq:Assumption}, we have $\P_x( O_n \in \cdot) \stackrel{d}{\to} \pi_+$ as $n \to \infty$ for any $x \in \ZZ_0$.
\end{conj}

Below we discuss the difficulties of proving convergence of $O_n$ in other metrics on probability distributions under the minimal assumptions in \eqref{eq:Assumption}. Let us start with two observations.

First, the total variation norm is clearly inappropriate since it requires the
spread out assumption, as explained in the beginning of Section~\ref{sec: 
Convergence}. Moreover, in the non-spread out non-arithmetic case the chain of
overshoot is not $\psi$-irreducible and thus not Harris recurrent,
placing it outside of the scope of the well-established classical convergence theory (see Meyn and
Tweedie~\cite{MeynTweedie}).  In fact, the spread out assumption
on the distribution of $X_1$ is equivalent to $\psi$-irreducibility of $O$ (and
$S$, of course). To see this, recall that any $\psi$-irreducible Markov chain
on $\R$ has a finite period $p$ by Theorems 5.2.2 and 5.4.4 in Meyn and
Tweedie~\cite{MeynTweedie}. Then, by Theorem 4 in Roberts and
Rosenthal~\cite{RobertsRosenthal2004}, which we used in the proof of
Theorem~\ref{thm: convergence general}, the $\psi$-irreducibility of $O$
implies that the aperiodic chain $(O_{pn})_{n \ge 0}$ converges to $\pi_+$ in the
total variation distance. But this can only be true when the distribution of
$X_1$ is spread out. 

Second, recall from Section~\ref{sec:Stationary_distribution_via_reversibility} that stationarity of $\pi_+$ for the chain $O$ can be established by factorizing the transition kernel of $O$ into the Markov kernels $P$ and $Q$, defined
in~\eqref{eq:P_and_Q_kernels}, both having $\pi_+$ as their stationary
distribution (see~\eqref{eq:O_U_kernels}). Unfortunately, this representation appears to be of a very limited use for studying the questions of convergence. In fact, the following example shows that the chain generated by $Q$ may
have an invariant distribution other that $\pi_+$, hence it may fail to converge to $\pi_+$ starting from an arbitrary point.
\begin{example} \label{ex: no uniqueness}
Let $X_1$ satisfy
$\P(X_1=a|X_1>0)=1$ for some $a>d$.
Then for any $x \in (0, a)$ we have $Q(x, dy) =\delta_{a-x}(dy)$
and hence $\frac12 \delta_x + \frac12 \delta_{a-x}$  is a stationary distribution of $Q$.
An analogous phenomenon occurs for any
non-arithmetic distribution of $X_1$ whose restriction to $\ZZ_0^+$ is atomic
with finitely many atoms.  
\end{example}

The next candidate is convergence in $L^2(\pi_+)$. First of all, here we can
work only with initial distributions (of $O_0=S_0$) that are absolutely continuous
with respect to $\pi_+$. Given that the transition operator of the chain of overshoots $O$ is the product of two reversible transition operators 
(see Section~\ref{sec:Stationary_distribution_via_reversibility} above), 
it is tempting to apply the methods of the theory of self-adjoint operators. We would need to
show that either $P$ or $Q$ has a spectral gap. A plausible way to prove this
is to check that the operator is compact, with $1$ being an eigenvalue of
multiplicity one, and that $-1$ is not an eigenvalue. 

The operator $Q$ appears to be more amenable for the analysis, but it seems that 
$Q$ may be non-compact for a general distribution of increments. In addition,
Example~\ref{ex: no uniqueness} above shows that $1$ can be a multiple
eigenvalue of $Q$, since the $Q$-chain can in general have more than one
stationary distribution on $\ZZ_d^+$. We are not aware of any works that
establish compactness of Markov transition operators on infinitely-dimensional
functional spaces without assuming some form of absolute continuity (as in this paper with spread out distributions of increments).

Regarding the weak convergence of Markov chains, the only technique we are
aware of is based on the so-called $\varepsilon$-coupling for continuous-time
Markov chains; see Thorisson~\cite[Section~5.6]{Thorisson}. This does not
seem to be applicable in the non-arithmetic case: 
even though, 
for any distinct real values $x_1$ and $x_2$, the walks $x_1 + S'$ and $x_2+ S'$ 
enjoy a version of $\varepsilon$-coupling (see Thorisson~\cite[Theorem~2.7.1]{Thorisson}),
the level zero will be crossed at different times by the two walks making it hardly possible to
deduce that the corresponding chains of overshoots are eventually only a small distance away from each other.

Our last candidate are Wasserstein-type metrics with a carefully chosen
distance on $\ZZ_0^+=[0, \infty)$. Here there is a promising approach,
introduced by Hairer and Mattingly~\cite{Hairer+2011,HairerMattingly}, which
works under a significantly relaxed version of the restrictive
$\psi$-irreducibility assumption and allows one to prove convergence of Markov
chains whose transition probabilities can even be mutually singular. Our
problem with non-arithmetic distributions that are not spread out appears to be
in this category, but we were unable to apply these ideas in our context because 
of the analytical intrectability of the transition kernel of the chain $O$.

\bibliographystyle{plain}
\bibliography{overshoot}

\section*{Acknowledgements}
AM is supported by the EPSRC grant EP/P003818/1 and a Fellowship at The Alan Turing Institute,
sponsored by the Programme on Data-Centric Engineering funded by Lloyd's Register Foundation.
AM is also supported by The Alan Turing Institute under the EPSRC grant EP/N510129/1.
This work was started when VV was affiliated to Imperial College London, where
he was supported by People Programme (Marie Curie Actions) of the
European Union's Seventh Framework Programme (FP7/2007-2013) under REA grant
agreement n$^\circ$[628803]. VV is supported in part by the RFBI Grant 19-01-00356.
\end{document}